\documentclass{amsart}

\usepackage{amsaddr}
\usepackage[utf8]{inputenc}
\usepackage{graphicx}
\usepackage{epsfig}
\usepackage{amsmath}
\usepackage{amssymb}
\usepackage{bbm}
\usepackage{mathtools}
\usepackage{amsthm}
\usepackage[all]{xy}
\usepackage{booktabs}
\usepackage{stmaryrd}
\usepackage{url}
\usepackage{longtable}
\usepackage[figuresright]{rotating}
\usepackage{amsfonts}
\usepackage{enumerate}
\usepackage{comment}
\usepackage{tikz}
\usepackage{xparse}
\usepackage{xstring}
\usepackage{pgffor}
\usepackage{hyperref}
\usepackage{ dsfont }
\usepackage[capitalize]{cleveref}

\usetikzlibrary{arrows,decorations.pathmorphing,decorations.pathreplacing,positioning,shapes.geometric,shapes.misc,decorations.markings,decorations.fractals,calc,patterns}
\tikzset{>=stealth',
         cvertex/.style={circle,draw=black,inner sep=1pt,outer sep=3pt},
         vertex/.style={circle,fill=black,inner sep=1pt,outer sep=3pt},
         star/.style={circle,fill=yellow,inner sep=0.75pt,outer sep=0.75pt},
         tvertex/.style={inner sep=1pt,font=\scriptsize},
         gap/.style={inner sep=0.5pt,fill=white}}
         
\newcounter{sarrow}

\DeclareFontFamily{OT1}{pzc}{}
\DeclareFontShape{OT1}{pzc}{m}{it}{<-> s * [1.10] pzcmi7t}{}
\DeclareMathAlphabet{\mathpzc}{OT1}{pzc}{m}{it}

\DeclareMathOperator{\hhom}{Hom}

\DeclareMathOperator{\eend}{End}

\DeclareMathOperator{\Gal}{Gal}

\DeclareMathOperator*{\im}{im}

\DeclareMathOperator{\Tate}{Tate}

\newcommand\set[1]{\mathbb{#1}}

\newcommand\clos[1]{\overline{#1}}

\newcommand\sh[1]{\mathcal{#1}}

\newcommand\units[1]{#1^{\times}}

\newcommand\idl[1]{\mathfrak{#1}}

\newcommand\mor[3][]{#2 \xrightarrow{#1} #3}
\newcommand\injmor[3][]{#2 \xhookrightarrow{#1} #3}
\newcommand\surjmor[3][]{#2 \xtwoheadrightarrow{#1} #3}
\newcommand\isomor[3][]{#2 \xrightarrow[#1]{\sim} #3}
\newcommand\inv[1]{#1^{-1}}

\newcommand{\Hom}[3][]{\hhom_{#1}\left(#2, #3\right)}

\newcommand{\End}[2][]{\eend_{#1}\left(#2\right)}

\newcommand{\Zmod}[2]{\set{Z}/#1^{#2}\set{Z}}

\newcommand{\order}[1]{\sh{O}_{#1}}

\newcommand{\Q}[1][]{\set{Q}_{#1}}
\newcommand{\F}[1][]{\set{F}_{#1}}
\newcommand{\Z}[1][]{\set{Z}_{#1}}
\newcommand{\C}[1][]{\set{C}_{#1}}

\newcommand{\GmcR}[1]{\widehat{\set{G}}_{m/R}}

\newcommand{\Qc}[1][]{\clos{\set{Q}}_{#1}}
\newcommand{\Fc}[1][]{\clos{\set{F}}_{#1}}
\newcommand{\Zc}[1][]{\clos{\set{Z}}_{#1}}

\DeclareMathOperator{\GL}{\mbox{GL}}

\DeclareMathOperator{\Aut}{\mbox{Aut}}

\newcommand{\xtwoheadrightarrow}[2][]{%
  \xrightarrow[#1]{#2}\mathrel{\mkern-14mu}\rightarrow
}

\newcommand\mat[4]{\begin{pmatrix}#1 & #2 \\ #3 & #4 \end{pmatrix}}
\newcommand\cyc[3]{\left\langle #1, #2, #3 \right\rangle}

\newcommand\dbr[1]{\left\llbracket #1 \right\rrbracket}

\newtheorem{theorem}{Theorem}[section]
\newtheorem{remark}[theorem]{Remark}

\newtheorem{conjecture}[theorem]{Conjecture}

\newtheorem{lemma}[theorem]{Lemma}
\newtheorem{proposition}[theorem]{Proposition}
\newtheorem{definition}[theorem]{Definition}

\newtheorem{question}[theorem]{Question}

\usepackage{tikz-cd}

\usepackage{chngcntr}
\usepackage{apptools}
\AtAppendix{\counterwithin{theorem}{section}}

\title{Congruences modulo prime powers of Hecke eigenvalues in level $1$}
\author{Nadim Rustom}
\address{
National Center for Theoretical Sciences\\
Mathematics Division\\
No. 1, Sec. 4, Roosevelt Rd., Taipei City 106, Taiwan Room 203, Astronomy-Mathematics Building, National Taiwan University}
\email{restom.nadim@gmail.com}


\begin{document}

\begin{abstract} We continue the study of strong, weak, and $dc$-weak eigenforms introduced by Chen, Kiming, and Wiese. We completely determine all systems of Hecke eigenvalues of level $1$ modulo $128$, showing there are finitely many. This extends results of Hatada and can be considered as evidence for the more general conjecture formulated by the author together with Kiming and Wiese on finiteness of systems of Hecke eigenvalues modulo prime powers at any fixed level. We also discuss the finiteness of systems of Hecke eigenvalues of level $1$ modulo $9$, reducing the question to the finiteness of a single eigenvalue. Furthermore, we answer the question of comparing weak and $dc$-weak eigenforms and provide the first known examples of non-weak $dc$-weak eigenforms. 
\end{abstract}

\maketitle


\section{Introduction}
\subsection{Motivation}
The connection between modular forms and representations of the absolute Galois group $G_{\Q} = \Gal(\Qc/\Q)$ is one of the most active areas of research in modern number theory. Given a Hecke eigenform $f$ of a certain weight and level, Deligne (\cite{De71}) and Deligne-Serre (\cite{DS74}) showed that one can attach to $f$ a $2$-dimensional $p$-adic Galois representations for every prime $p$. The following years witnessed significant progress on the converse question: given a $2$-dimensional $p$-adic Galois representation, is it modular (i.e.\ does it come from a Hecke eigenform)?

A fundamental conjecture in this area is the Fontaine-Mazur conjecture (\cite{FM95}) which states that any $2$-dimensional $p$-adic Galois representation which ``resembles" (in some precise sense) a representation occurring in the \'etale cohomology of an algebraic variety is indeed modular up to twist. In recent years, there has been considerable progress on this conjecture due to Emerton (\cite{Em11a}) and Kisin (\cite{Ki09}).

Similarly, Serre's modularity conjecture states that any continuous, odd, and absolutely irreducible mod $p$ representation $\clos{\rho}: G_{\Q} \rightarrow \GL_2(\Fc[p])$ is modular  and gives precise recipes for the optimal weight and level of the corresponding eigenform. Serre's conjecture is now a theorem thanks to Khare and Wintenberger (\cite{KW10}). 

The question of what happens ``in between", i.e.\ for $2$-dimensional mod $p^m$ Galois representations, is then a very natural question to ask. This motivates the study of eigenforms mod $p^m$. Chen, Kiming, and Wiese started this study in \cite{CKW13} and introduced three progressively weaker notions of eigenforms mod $p^m$: strong eigenforms, weak eigenforms, and $dc$-weak eigenforms (``dc" stands for divided congruences, a notion introduced by Katz). These are all defined as the mod $p^m$ reductions of elements in Katz's space of divided congruences. A $dc$-weak eigenform is the mod $p^m$ reduction of a divided congruence $f$ such that $T_n(f) \equiv a_n(f)f \pmod{p^m}$ for all $n \geq 1$. A weak eigenform is a $dc$-weak eigenform that is the mod $p^m$ reduction of a modular form. A strong eigenform is the mod $p^m$ reduction of a classical eigenform in characteristic $0$. As was shown in \cite{CKW13}, one can naturally attach Galois representations to mod $p^m$ $dc$-weak eigenforms, and these satisfy a Ribet-type level lowering result. Conversely, Tsaknias and Wiese showed in \cite{TW16} that mod $p^m$ Galois representation satisfying certain conditions arise from $dc$-weak eigenforms. Of particular interest are the strong eigenforms, as the corresponding mod $p^m$ Galois representations attached to them can be seen as successive $p$-adic approximations to the Galois representations attached to classical eigenforms in characteristic $0$. 

\subsection{Congruences between modular forms} The quest for congruences satisfied by particular modular forms can be traced back to the work of Ramanujan (\cite{Ra16}) and has occupied several authors in the 20th century (e.g.\ \cite{BC47}, \cite{Ko62}, \cite{As68}, and others). Later, attention turned towards finding congruences satisfied by whole collections of modular forms of varying level and weight (e.g.\ the results of Hatada cited below). The present work falls within the latter camp. We will now explain this point of view. 

Let $p$ be a prime, and fix algebraic closures $\Qc$ of $\Q$ and $\Qc[p]$ of $\Q[p]$ and a commutative diagram of embeddings
\[\begin{tikzcd} \Q \arrow[hook]{d} \arrow[hook]{r} & \Q[p] \arrow[hook]{d} \\ \Qc \arrow[hook]{r}& \Qc[p]. \end{tikzcd}    \]
Let $v_p$ be the normalised valuation ($v_p(p)=1$) on $\Qc[p]$ and $\Zc[p]$ the ring of integers in $\Qc[p]$. Fix an integer $N \geq 1$ and consider the $\Zc[p]$-module $S(N, \Zc[p])$ spanned by all modular forms of level $N$ (i.e.\ on $\Gamma_1(N)$) with coefficients in $\Zc[p]$ of all weights. We will say that two modular forms $f, g \in S(N, \Zc[p])$ are congruent modulo $p^m$, and write $f \equiv g \pmod{p^m}$, if $v_p(a_n(f) - a_n(g)) > m - 1$ for all $n \geq 1$. Note that this is not the same as having the $q$-expansion of $f-g$ be divisible by $p^m$. When we mean the latter case, we will write $f \equiv g \pmod{p^m\Zc[p]}$. In this paper, ``eigenform" will only refer to normalised ($a_1 = 1$) cuspidal Hecke eigenforms. 

The following theorem is classical.
\begin{theorem}[Jochnowitz (\cite{Jo82}), Serre-Tate]\label{finitemodp} There are only finitely many congruence classes $\pmod{p}$ of eigenforms of level $N$. Any eigenform of level $N$ is congruent mod $p$ to an eigenform of weight at most $p^2 + p$, and congruent away from $p$ to an eigenform of weight at most $p + 1$. 
\end{theorem}

The first statement in \cref{finitemodp} is an instance of a ``finiteness result" for mod $p^m$ eigenforms, while the second statement is an instance of a ``weight bound" result. In \cite{KRW16} the author, together with Kiming and Wiese, studied a generalisation of \cref{finitemodp} and showed, with the help of Frank Calegari, that a weight bound result holds as well for strong eigenforms mod $p^m$.

\begin{theorem} There exists a constant $\kappa(N, p, m)$ depending only on $N, p$, and $m$ such that any eigenform of level $N$ is congruent mod $p^m$ to a modular form (not necessarily an eigenform) of level $N$ and weight at most $\kappa(N, p, m)$. 
\end{theorem} 

The author showed in \cite{Ru17} that weight bounds do not exist for $dc$-weak eigenforms in general and expects that they do not exist even for weak eigenforms. 

On the other hand, the question of whether the mod $p$ finiteness result in \cref{finitemodp} can be generalised to higher prime powers seems much more difficult. The author, together with Kiming and Wiese, made the following conjecture in \cite{KRW16}.

\begin{conjecture}[Finiteness conjecture]\label{finitenessconjecture} For any $m \geq 1$, there are only finitely many congruence classes $\pmod{p^m}$ of eigenforms of level $N$. 
\end{conjecture}
Such a finiteness statement does not hold for $dc$-weak eigenforms (and not even for weak eigenforms), as was first shown in \cite{CE04}.

The evidence for \cref{finitenessconjecture} is not plentiful, but we list what is known. First, there is the following result of Hatada.
\begin{theorem}[\cite{Ha79}]\label{hatada2} Let $f$ be an eigenform of level $1$. Then $a_2(f)\equiv 0 \pmod{8\Zc[2]}$ and $a_\ell(f) \equiv 1 + \ell \pmod{8\Zc[2]}$ for all odd primes $\ell$. In particular, $f \equiv \Delta\pmod{8\Zc[2]}$, where $\Delta$ is the unique normalised cuspform of weight $12$ and level $1$. 
\end{theorem}
Note that \cref{hatada2} is actually stronger than what is implied by \cref{finitenessconjecture}. Hatada also proved the following, which is also stronger than what is implied by \cref{finitemodp}.
\begin{theorem}[\cite{Ha79}]\label{hatada3} Let $f$ be an eigenform of level $1$. Then $a_3(f) \equiv 0 \pmod{3\Zc[3]}$ and $a_\ell(f) \equiv 1+\ell \pmod{3\Zc[3]}$ for all primes $\ell \not = 3$. In particular, $f \equiv \Delta\pmod{3\Zc[3]}$.
\end{theorem}
Hatada's proof relies on the study of the action of Hecke operators on lattices generated by periods of modular forms. Aspects of his argument have been formalised in terms of geometry and cohomology in \cite{CE04}.

Buzzard (\cite{Bu05}) has also investigated related questions. For each eigenform $f \in S(N, \Qc[p])$, let $K_f := \Q[p][\{a_\ell(f) : \ell \nmid Np \}]$. 
\begin{question}[Buzzard]\label{buzzardquestion} Suppose $p \nmid N$. Does there exists a constant $B = B(N, p)$, depending only on $N$ and $p$, such that $[K_f:\Q[p]] < B(N,p)$ for all eigenforms $f$ of level $N$?
\end{question}
In light of \cref{finitemodp}, this is the same as asking whether the ramification index of $p$ in the fields $K_f$ is bounded independently of $f$. Kilford obtained the following result.
\begin{theorem}[\cite{Ki04}] Let $f$ be an eigenform of level $4$ and odd weight. Then $a_n(f) \in \Q[2]$ for all $n \geq 1$. 
\end{theorem}
The author has personally checked that the primes $7,11, 17, 29, 53$, and $61$ are unramified in $K_f$ for all eigenforms $f$ of level $1$ and of weight $\leq 530$ and that $3$ is unramified for all $k \leq 1000$. Buzzard's question and the finiteness conjecture \cref{finitenessconjecture} are connected by the following result. Let $\bf{B}$ be the statement ``the answer to Buzzard's question is yes", and $\bf{Fin}$ the statement ``The finiteness conjecture \cref{finitenessconjecture} is true". In \cite{KRW16}, the following result is shown.
\begin{theorem} $\bf{B} \Leftrightarrow \bf{Fin} \mbox{ } + \bf{I}$ where $I$ is a collection of index conjectures (see \cite{KRW16}, \S2.3, for further details).  
\end{theorem}

Finally, we state a conjecture due to Coleman and Stein (\cite{CS04}) which is a precise formulation of a special case of \cref{finitenessconjecture}.

\begin{conjecture}[Coleman-Stein]\label{colemansteinconjecture} There are exactly five residue classes in $(\Zmod{9}{})\dbr{q}$ of normalised eigenforms in $S_k(\Gamma_0(N))$ where $k \geq 1$ and $N = 1,3, 9$. They all appear in level $1$ and are given in the following table.

{\label{cstable}\rm\tt \begin{center}\begin{tabular}{|c|l|}\hline
{\rm Weight} & {\rm [ $a_2, a_3, \ldots, a_{43} \mod 9\Zc[3]$ ]}\\\hline
 12    & [ 3, 0, 6, 5, 3, 8, 0, 2, 6, 3, 8, 2, 6, 5 ]\\
 16    & [ 0, 0, 0, 2, 0, 2, 0, 2, 0, 0, 2, 2, 0, 2 ]\\
 20    & [ 6, 0, 3, 8, 6, 5, 0, 2, 3, 6, 5, 2, 3, 8 ]\\
 24    & [ 6, 0, 3, 5, 6, 8, 0, 2, 3, 6, 8, 2, 3, 5 ]\\
 32    & [ 3, 0, 6, 8, 3, 5, 0, 2, 6, 3, 5, 2, 6, 8 ]\\
\hline\end{tabular}\end{center}}
\end{conjecture}
Coleman and Stein verified \cref{colemansteinconjecture} up to weight $74$ in level $1$ and weight $40$ in levels $3$ and $9$. This conjecture was discussed in \cite{CE04}, but the authors remarked that their methods were not able to yield a proof.

This paper is a continuation of the study of strong, weak, and $dc$-weak eigenforms modulo prime powers and the finiteness conjecture (\cref{finitenessconjecture}).

\subsection{Results and strategy} 

We will consider, for any finitely generated $\Z[p]$-algebra $R$, the spaces $D(N, R)$ of divided congruences of level $N$ and coefficients in $R$ defined by Katz in \cite{Ka75}. Katz's theory of divided congruences, the corresponding Hecke action, and the notions of strong, weak, and $dc$-weak eigenforms with coefficients in a ring $R$ will be recalled in \cref{section:modformdc} and \cref{section:swdc}. We will then turn to the question of comparing weak and $dc$-weak eigenforms. This question was raised in \cite{CKW13} and was given a partial answer in \cite{Ru17}. In \cref{section:compweakdcweak}, we will give a complete answer. We will show that a single Hecke operator $t_\Lambda$, corresponding to a weight twist, is able to identify weak eigenforms. 

\newtheorem*{mainthmIII}{Theorem \ref{dcweakisweak}}
\begin{mainthmIII} Suppose $p \geq 3$. Let $K$ be a finite extension of $\Q[p]$, $\order{K}$ its ring of integers, and $m \geq 1$ an integer. Suppose $f \in D(N,\order{K}/\pi_K^m \order{K})$ a $dc$-weak eigenform. Then $f$ is weak if and only if the eigenvalue of $t_\Lambda$ corresponding to $f$ lies in the image of $\Z[p]$ in $\order{K}/\pi_K^m \order{K}$. 
\end{mainthmIII}

\cref{dcweakisweak} does not hold when $p = 2$ and $N = 1$. In \cref{section:wefmod4}, we calculate all $dc$-weak eigenforms of level $1$ with coefficients in $\Zmod{4}{}$. Calculating these weak eigenforms is made possible by the existence of explicit weight bounds for weak eigenforms with coefficients in $\Zmod{2}{m}$ that were obtained in \cite{KRW16} using Nicolas-Serre theory (\cite{NS12}). Curiously, this provides the first known examples of $dc$-weak eigenforms that are not weak. 

We use a similar argument in \cref{section:wefmod9} to calculate all $dc$-weak eigenforms of level $1$ with coefficients in $\Zmod{9}{}$. An analogue of Nicolas-Serre theory has not yet been developed for modular forms mod $3$, so we use results of \cite{BK15} as a substitute. 

Next, we come to the main theorem of this paper. 
\newtheorem*{mainthmI}{Theorem \ref{finiteness}}
\begin{mainthmI} If $f$ is a level $1$ eigenform of weight $k$ and $\ell$ is an odd prime then $a_\ell(f) \equiv 1+\ell^{k-1} \pmod{128 \Zc[2]}$. Consequently, there are only finitely many congruence classes $\pmod{128\Zc[2]}$ of eigenforms of level $1$. 
\end{mainthmI}

The method is not expected to work in its current form to deal with the question of finiteness mod $2^m$ for $m \geq 8$, or with other primes. A discussion of the limitations is given at the end of \cref{section:algorithm}.

We remark that Hatada (\cite{Ha81}, based on a suggestion of Serre, c.f.\ Remark 2 in \cite{Ha79}) asked the following question. 
\begin{question}\label{hatadaserrequestion}Is it true that $a_\ell(f) \equiv 1 + \ell \pmod{2^a\Zc[2]}$ for all eigenforms $f$ of level $1$, primes $\ell \equiv \pm 1 \pmod{2^{a-1}}$, and $a \leq 13$?
\end{question}
Hatada showed that the answer to \cref{hatadaserrequestion} is affirmative for all $a \leq 5$. \cref{finiteness} answers \cref{hatadaserrequestion} affirmatively for all $a \leq 7$. 

We will also reduce the level $1$ Coleman-Stein conjecture (\cref{colemansteinconjecture}) to a finiteness conjecture for a single coefficient.
\newtheorem*{a2conjecture}{\cref{a2conj}}
\begin{a2conjecture}Let $f$ be an eigenform of level $1$ and weight $w + 2$. Then 
\[ a_2(f) \equiv \begin{cases} 3\mbox{ or } 6\pmod{9\Zc[3]}\indent \mbox{if } w \equiv 0 \pmod{6} \\ 3 \mbox{ or } 6\pmod{9\Zc[3]}\indent \mbox{if } w \equiv 4 \pmod{6}\\ 0 \pmod{9\Zc[3]}\indent\indent\indent \mbox{if } w \equiv 2 \pmod{6}.\end{cases} \]
\end{a2conjecture}
\newtheorem*{mainthmII}{Theorem \ref{finitenessa2}}
\begin{mainthmII} \cref{a2conj} implies that there are only finitely many congruence classes $\pmod{9\Zc[3]}$ of eigenforms of level $1$.  
\end{mainthmII}

To prove \cref{finiteness} and \cref{finitenessa2}, we will identify explicit generators of the relevant Hecke algebras. In \cref{section:heckegen}, we exploit the fact that there is a one-to-one correspondence between $dc$-weak eigenforms with coefficients in $\Zmod{p}{2}$ and $\Z[p]$-algebra homomorphisms $\mor{\set{T}(p,N)}{\Zmod{p}{2}}$ (the $\Zmod{p}{2}$-valued points of the Hecke algebra) to identify these generators. We find that in order to prove \cref{finiteness}, it is enough to determine the congruence classes of the eigenvalues of $T_3$ and $T_5$ modulo $128\Zc[2]$. This will show that any eigenform of level $1$ is congruent away from $2$ to an eigenform of level $1$ and weight $k \leq 46$. The precise congruences in the statement of \cref{finiteness} are then shown for these eigenforms of small weight. Similarly, in order to prove \cref{finitenessa2}, it is enough to determine the congruence classes of the eigenvalues of $T_7$ modulo $9\Zc[3]$. We prove \cref{finiteness} and \cref{finitenessa2} in \cref{section:finiteness2} and \cref{section:finiteness3} assuming the congruences for the eigenvalues $a_2$ and $a_3$ (modulo $128\Zc[2]$) and $a_7$ (modulo $9\Zc[3]$). 

To prove these congruences, we use Merel's formulation of the theory of modular symbols (\cite{Me94}). In level $1$, this theory has a simple description, and the action of the relevant Hecke operators $T_3$, $T_5$, and $T_7$ is relatively easy to write down. Using Sage, we reduce the proof of these divisibility statements to verifying a finite set of polynomial identities in $\left(\Zmod{128}{}\right)[X,Y]$ and $\left(\Zmod{9}{}\right)[X,Y]$. In \cref{section:algorithm}, we describe the algorithm used to discover and prove these identities. 

The computational part of the proof shares the spirit of Hatada's and Calegari-Emerton's arguments. The main difference is that while they prove congruences working with all Hecke operators $T_\ell$ for $\ell$ varying in a congruence class, we work with specific Hecke operators $T_\ell$ for small primes $\ell$ and use extra knowledge about the Hecke algebra to deduce congruences for eigenforms. 


\section{Modular forms and divided congruences}\label{section:modformdc}
Let $N \geq 1$ be an integer and write $N = p^r N_0$ where $r \geq 0$ and $p \nmid N_0$. Let $k \geq 0$ be an integer. We will denote by $S_k(N, \Z)$ the $\Z$-submodule of $\Z[]\dbr{q}$ spanned by the $q$-expansions of cuspidal modular forms of weight $k$ and level $N$ (i.e.\ on $\Gamma_1(N)$) with $q$-expansion coefficients in $\Z$. 

Let $K$ be a finite extension of $\Q[p]$, $\order{K}$ its ring of integers, $\pi_K$ a uniformiser of its maximal ideal, and $\F[K] = \order{K}/\pi_K\order{K}$ its residue field. We will denote by $S_k(N, \order{K})$ the $\order{K}$-submodule of $\order{K}\dbr{q}$ spanned by the image of $S_k(N, \Z)$ via the canonical map. 
Let
\[ S_{\leq k}(N, \order{K}) := \sum_{i=0}^{k} S_{i}(N, \order{K}) = \bigoplus_{i=0}^{k}S_{i}(N, \order{K}),   \]
\[ S(N, \order{K}) := \sum_{i\geq 0} S_{i}(N, \order{K}) = \bigoplus_{i\geq 0}S_{i}(N, \order{K}),   \]
\[ D_{k}(N, \order{K}) :=  \{ f\in \order{K}\dbr{q}: \pi_K^t f \in S_{\leq k}(N, \order{K}) \mbox{ for some } t \geq 0  \}, \]
\[ D(N, \order{K}) := \bigcup_{k\geq 0} D_{k}(N, \order{K}).  \]
For a $p$-adically complete and separated $\order{K}$-algebra $R$, we let $S_k(N,R)$ (respectively, $S_{\leq k}(N,R)$, $S(N,R)$, $D_k(N,R)$, $D(N,R)$) denote the $R$-submodule of $R\dbr{q}$ spanned by the image of $S_k(N, \order{K})$ (respectively, of $S_{\leq k}(N, \order{K})$, $S(N,\order{K})$, $D_k(N, \order{K})$,  $D(N, \order{K})$) via the canonical map $\order{K}\dbr{q} \rightarrow R\dbr{q}$. 
We call $D_{k}(N, R)$ the module of divided congruences with coefficients in $R$ of weight at most $k$ and level $N$. We call $D(N, R)$ the module of divided congruences with coefficients in $R$ and level $N$. For simplicity, we will drop $N$ from the notation if $N = 1$. 

The modules $D_k(N,R)$ and $D(N, R)$ satisfy a nice base change property.
\begin{proposition}[\cite{Ru17}, Proposition 2.2]\label{dcbasechange} We have 
\[ D_k(N,R) = D_k(N, \order{K}) \otimes_{\order{K}} R    \]
and 
\[ D(N,R) = D(N, \order{K}) \otimes_{\order{K}} R.    \]
\end{proposition}

Divided congruences have a natural geometric interpretation in terms of trivialised elliptic curves. We recall this theory presented in \cite{Ka75}, \cite{Ka75a}, and \cite{Go88}. A trivialised elliptic curve over $R$ with a level $N = p^rN_0$ structure is a triple $(E/R, \iota_N, \varphi)$ consisting of 
\begin{enumerate}[(i)]
\item an elliptic curve $E$ over $R$, 
\item an $R$-isomorphism $\varphi: \isomor{\widehat{E}}{\widehat{\set{G}}_m}$ between the formal group of $E$ and the formal multiplicative group, and
\item an inclusion $\iota_N: \injmor{\mu_{N}}{E[N]}$ of finite flat group schemes over $R$ such that the induced composite map 
\[\injmor{\mu_{p^r}}{\mor[\varphi]{\widehat{E}}{\widehat{\set{G}}_m}}   \]
is the canonical inclusion. 
\end{enumerate}
An elliptic curve admitting a trivialisation is necessarily fibre-by-fibre ordinary, and the converse is true up to base-change. A morphism $\mor{(E, \iota_{N}, \varphi)}{(E', \iota_{N}',\varphi')}$ of trivialised elliptic curves over $R$ is a morphism $\alpha: \mor{E}{E'}$ of elliptic curves over $R$, compatible with the level structures $\iota_N$ and $\iota_N'$ and with the trivialisations $\varphi$ and $\varphi'$. The functor
\[ \sh{F}^{triv}: R \mapsto \{\mbox{isomorphism classes of trivialised elliptic curves}\]\[\mbox{ with level $N$ structure } (E/R, \iota_N, \varphi)\}  \]
on the category of $p$-adically complete and separated $\order{K}$-algebras is represented by a $p$-adically complete and separated $\order{K}$-algebra $\set{V}(N, \order{K})$. An element of $\set{V}(N,\order{K})$ is called a $p$-adic modular form of level $N$ and coefficients in $\order{K}$ and can be understood as a rule $f$ which assigns to any trivialised elliptic curve $(E, \iota_N, \varphi)$ over a $p$-adically complete and separated $\order{K}$-algebra $R$ a value $f(E, \iota_N, \varphi)\in R$ depending only on the $R$-isomorphism class of $(E, \iota_N, \varphi)$ and whose formation is compatible with base change. The Tate curve, defined over the $p$-adic completion $\widehat{\Z[p](\!( q )\!)}$ of $\Z[p](\!( q )\!)$, admits a canonical trivialisation $\varphi_{can}$ and a canonical level $N$ structure $\iota_{N, can}$, and the $q$-expansion of $f \in \set{V}(N,\order{K})$ is given by $f(q) = f(\Tate(q), \iota_{N, can}, \varphi_{can})$. We let $\set{V}^{par}(N, \order{K}) \subset \set{V}(N, \order{K})$ denote the set of elements $f \in \set{V}(N, \order{K})$ such that \[f(\Tate(q), \iota_N, \varphi_{can}) \in q\order{K}\dbr{q}\] for all level $N$ structures on $\Tate(q)$.  

The following theorem is key to the whole theory.
\begin{theorem}\label{dckey} The $q$-expansion map 
\[ \mor{\set{V}^{par}(N, \order{K})}{\order{K}\dbr{q}}  \]
is injective and its image contains the module $D(N, \order{K})$ of divided congruences as a dense subset. 
\end{theorem} 
\begin{proof} This is shown in \cite{Ka75} when $N \geq 3$. The condition on the level is present because in that paper, Katz chooses to work with moduli schemes. However, in \cite{Ka75a}, Katz points out how to go around this and work on moduli stacks in order to obtain the same result without any restriction on the level. The injectivity of the $q$-expansion map is a consequence of the irreducibility of the stack of trivialised elliptic curves. See for example \cite{Ge04} where the case $p = 2$, $N = 1$ is worked out explicitly. 
\end{proof}
In light of \cref{dckey}, we may view $\set{V}^{par}(N, \order{K})$ as the $p$-adic completion of the module $D(N, \order{K})$ of divided congruences. 

Each element $(x,y) \in \units{\Z[p]}\times\units{\left(\Zmod{N}{}\right)} = \Aut(\widehat{\set{G}}_m)\times\units{\left(\Zmod{N}{}\right)}$ acts on elements $f \in \set{V}^{par}(N, \order{K})$ by
\[ \langle x,y\rangle f(E, \iota_N, \varphi) = f(E, y\iota_N, \inv{x}\varphi).    \] We will denote the operator $\langle x, 1\rangle$ by $[x]$. Because of \cref{dckey}, we get an action of $\units{\Z[p]}$ on $D(N, \order{K})$ given by
\[ f = \pi_K^{-t} \sum_i f_i \mapsto [x]f =  \pi_K^{-t} \sum_i x^{k_i}f_i.   \]
Katz showed that
\begin{theorem}\label{katzdcmodp} The subspace $S(N, \F[K])$ of $D(N, \F[K])$ is precisely the set of elements invariant under the action of $1+p\Z[p]$.
\end{theorem}

\begin{remark} When $N = 1$, there is an isomorphism $(E/R, \varphi) \cong (E/R, -\varphi)$ induced by the involution ``multiplication by -1" on $E$. Thus the operator $[-1]$ acts trivially on $p$-adic modular forms of level $1$. This can be understood at the level of divided congruences, as they all come at level $1$ from modular forms of even weight. In particular, when $p = 2$, we have \[ \units{\Z[2]} = 1 + 2\Z[2] = \left(1 + 4\Z[2] \right)\times \langle -1\rangle,  \]
so $\units{\Z[2]}$ acts through its quotient $1+4\Z[2]$. 
\end{remark}


\section{Strong, weak, and $dc$-weak eigenforms}\label{section:swdc}

The action of the Hecke operators $T_n$, $p\nmid n$, can be extended from $S(N,\order{K})$ to $D(N, \order{K})$ (see \cite{BK15}). If $p | N$, we can also extend the action of the $U$ operator from $S(N, \order{K})$ to $D(N, \order{K})$. This action is given by
\[a_n(Uf) = \begin{cases}0 \indent \mbox{if } p \nmid n, \\ a_{n/p}(f) \indent \mbox{if } p | n. \end{cases}   \]

For each $m \geq 1$, let $R_m := \order{K}/\pi_K^{m}\order{K}$ (so $R_1 = \F[K]$). By \cref{dcbasechange}, the Hecke operators $T_n$, $p\nmid n$, and $U$ if $p | N$, induce operators on $D(N, R_m)$ which are compatible with their action on $q$-expansions. 

\begin{lemma} The operator $U$ induces an operator on $D(N, R_m)$ which is compatible with the action of $U$ on $q$-expansions. 
\end{lemma}
\begin{proof} This is clear by the above discussion if $p | N$. If $p \nmid N$, we proceed as follows. By \cref{dckey} we can identify $\set{V}^{par}(N, \order{K})$ with the $p$-adic completion of $D(N, \sh{O}_K)$. By \cite{Go88}, Proposition I.3.9, $\set{V}^{par}(N, \order{K})$ is also the $p$-adic completion of $D(Np, \order{K})$. Thus by continuity we can extend the action of $U$ from $D(Np, \order{K})$ to $\set{V}^{par}(N, \order{K})$. Since the image of $\pi_K$ in $R_m$ is nilpotent, we have 
\[D(N, R_m) = D(N, \order{K})\otimes_{\order{K}} R_m = \set{V}^{par}(N, \order{K})\otimes_{\order{K}} R_m \]
by \cref{dcbasechange}. The action of $U$ on $\set{V}^{par}(N, \order{K})$ then induces an action on $D(N, R_m)$. \end{proof}

Let $\set{T}^{pf}(N, R_m)$ denote the partially full Hecke algebra on $D(N, R_m)$, i.e.\ the $R_m$-subalgebra of $\End[R_m]{D(N, R_m)}$ generated by $\{T_n : p \nmid n\}$. Let $\set{T}(N, R_m)$ denote the $R$-subalgebra of $\End[R_m]{D(N, R_m)}$ generated by $\set{T}^{pf}(N, R_m)$ and $U$. We call $\set{T}(N,R_m)$ the full Hecke algebra on $D(N,R_m)$. Just like in the previous section, we drop $N$ from the notation if $N = 1$. 

The algebra $\set{T}(N, R_m)$ satisfies a nice base change property.

\begin{proposition}[\cite{Ru17}, Corollary 3.5]\label{heckebasechange} We have
\[ \set{T}(N, R_m) = \set{T}(N, \order{K})\otimes_{\order{K}} R_m.   \]
\end{proposition}

Define the pairing
\[ \mor{\set{T}(N, R_m) \times D(N, R_m)}{R_m},   \]
\[ (T, f) \mapsto a_1(Tf). \]
\begin{proposition}[\cite{Ru17}, Proposition 3.6] The pairing defined above induces an isomorphism
\[ D(N, R_m) \cong \Hom[cont]{\set{T}(N, R_m)}{R_m}  \]
where $\Hom[cont]{\set{T}(N, R_m)}{R_m}$ is the set of continuous\footnote{The topology on $\set{T}(N, R_m)$ is the projective limit topology induced by writing $\set{T}(N, R_m) = \varprojlim_k \set{T}_k(N,R_m)$ where $\set{T}_k(N,R_m)$ is the Hecke algebra (with the discrete topology) acting on $D_k(N,R_m)$.} $R_m$-linear maps $\mor{\set{T}(N,R_m)}{R_m}$.  
\end{proposition}

We now give two definitions the notions of strong, weak, and $dc$-weak eigenforms. The equivalence of these two definitions is guaranteed by the discussion above.

\begin{definition}\label{eigendefnI} \leavevmode
\begin{enumerate}
\item A (normalised\footnote{Note that this implies that $a_1(f) = 1$ since the operator $T_1$ is the identity.}) $dc$-weak eigenform with coefficients in $R_m$ of level $N$ is an element $f \in D(N, R_m)$ such that $T_n f = a_n(f)f$ for all $n \geq 1$, $p \nmid n$, and $Uf = a_p(f) f$. 
\item A $dc$-weak eigenform $f$ with coefficients in $R_m$ of level $N$ is called weak if there exists $k \geq 0$ (not necessarily unique) such that $f \in S_k(N, R_m)$.
\item A weak eigenform $f$ with coefficients in $R_m$ of level $N$ is called strong if there exists $k\geq 0$ (not necessarily unique), a finite extension $\order{L}$ of $\order{K}$, and an $\order{K}$-algebra homomorphism $\iota: \mor{\order{L}}{R_m}$ inducing a map $\iota^\ast: \mor{S_k(N, \order{L})}{S_k(N, R_m)}$ such that $f = \iota^\ast(f')$ for some normalised eigenform $f' \in S_k(N, \order{L})$.  
\end{enumerate}
\end{definition}
\begin{definition}\label{eigendefnII} \leavevmode
\begin{enumerate}
\item A $dc$-weak eigenform with coefficients in $R_m$ of level $N$ is an $R_m$-algebra homomorphism $\mor[\phi]{\set{T}(N, R_m)}{R_m}$.
\item A $dc$-weak eigenform $\phi$ with coefficients in $R_m$ of level $N$ is called weak if there exists $f \in S_k(N, R_m)$ for some $k \geq 0$ (not necessarily unique) and $f \in S_k(N, R_m)$ such that $\phi(T_n) = a_n(f)$ for all $n$ such that $p \nmid n$ and $\phi(U) = a_p(f)$. 
\item A weak eigenform with coefficients in $R_m$ of level $N$ is called strong if there exists $k \geq 0$ (not necessarily unique), a finite extension $\order{L}$ of $\order{K}$, an $\order{K}$-algebra homomorphism $\iota: \mor{\order{L}}{R_m}$ inducing a map $\iota^\ast: \mor{S_k(N, \order{L})}{S_k(N, R_m)}$, and $f \in S_k(N, R_m)$ such that $f = \iota^\ast(f')$ for some normalised eigenform $f' \in S_k(N, \order{K})$, $\phi(T_n) = a_n(f)$ for all $n$ such that $p\nmid n$ and $\phi(U) = a_p(f)$. 
\end{enumerate}
\end{definition}

For $dc$-weak eigenforms with coefficients in a finite field (i.e.\ when $m = 1$), we have the following result.
\begin{lemma}[Deligne-Serre lifting lemma]\label{dslifting} Every $dc$-weak eigenform in $D(N, \F[K])$ is strong. 
\end{lemma}
\begin{proof} See Lemme 6.11 of \cite{DS74} and Lemma 16 of \cite{CKW13}.
\end{proof}

\begin{remark}
Clearly, a strong eigenform is weak, and a weak eigenform is $dc$-weak, but these notions are not equivalent when $m \geq 2$. In \cref{section:wefmod4}, we will see that there are 16 $dc$-weak eigenforms of level $1$ with coefficients in $\Zmod{4}{}$, $8$ of which are weak. By \cref{hatada2}, only one of these is a strong eigenform. See also the discussion at the end of \S 3 in \cite{Ru17}. 
\end{remark}

The partial and full Hecke algebras are both semilocal complete and separated rings. Semilocality (i.e.\ having finitely many maximal ideals) can be deduced from \cref{finitemodp}. The maximal ideals of $\set{T}(N, \Z[p])$ are in one-to-one correspondence with eigenforms $\varphi:\mor{\set{T}(N,\Z[p])}{\F[p]}$ which in turn are in one-to-one correspondence with pairs $(\idl{m}, \lambda)$ of maximal ideals $\idl{m}$ of $\set{T}^{pf}(N, \Z[p])$ and eigenvalues $\varphi(U) = \lambda$ of $U$. Thus we have a decomposition
\[ \set{T}(N, \Z[p]) = \prod_{(\idl{m},\lambda)} \set{T}(N, \Z[p])_{\idl{m}, \lambda}, \indent  \set{T}^{pf}(N, \Z[p]) = \prod_{\idl{m}} \set{T}^{pf}(N, \Z[p])_{\idl{m}}\]
where $\set{T}^{pf}(N, \Z[p])_{\idl{m}}$ is the localisation of $\set{T}^{pf}(N, \Z[p])$ at the maximal ideal $\idl{m}$, $\set{T}(N, \Z[p])_{\idl{m},\lambda}$ is the localisation of $\set{T}(N, \Z[p])$ at the maximal ideal corresponding to the pair $(\idl{m}, \lambda)$.
\begin{proposition}\label{comparehecke}\leavevmode \begin{enumerate}[(i)]
\item For every maximal ideal $\idl{m}$ of $\set{T}^{pf}(N, \Z[p])$, there is a natural isomorphism of $\set{T}^{pf}(N, \Z[p])$-algebras \[\set{T}^{pf}(N, \Z[p])_{\idl{m}}\dbr{U}\cong \set{T}(N, \Z[p])_{\idl{m}, 0}.\]
\item Let $N = p^r N_0$ with $p \nmid N_0$. The algebras $\set{T}^{pf}(N, \Z[p])$ and $\set{T}^{pf}(N_0, \Z[p])$ are naturally isomorphic. 
\end{enumerate}
\end{proposition}
\begin{proof}\leavevmode
\begin{enumerate}[(i)]
\item \cite{Deo17}, Proposition 17.
\item This follows from \cite{Go88}, Proposition I.3.9. See \cite{Deo17}, Corollary 13 and \cite{BK15}, Corollary 13.
\end{enumerate}
\end{proof}

\section{Comparing weak and $dc$-weak eigenforms}\label{section:compweakdcweak}
For later use in this section and in the rest of the paper, we let 
\[E_k := 1 - \frac{2k}{B_k} \sum_{n \geq 1} \sigma_{k-1}(n) q^n = 1 - \frac{2k}{B_k} \sum_{n \geq 1} \sum_{d|n} d^{k-1} q^n\] for each even integer $k \geq 4$ be the Eisenstein series of level $1$ and weight $k$ where $B_k$ is the $k$th Bernoulli number. We also recall that, as a consequence of the Clausen-Von Staudt theorem (\cite{Se73}, \S 1.1 (d)), the coefficients of $E_k$ are $p$-integral whenever $k \equiv 0 \pmod{p-1}$ and, for each integer $m \geq 1$,
\[ E_k \equiv \begin{cases}1 \pmod{p^m} \indent \mbox{if } k \equiv 0 \pmod{p^{m-1}(p-1)} \mbox{ and } p \geq 3, \\ 1 \pmod{2^m} \indent \mbox{if } k \equiv 0 \pmod{2^{m-2}}.\end{cases}  \]

For each $m \geq 1$, define $\gamma(m)$ to be the positive integer such that $\Zmod{p}{\gamma(m)}$ is the image of $\Z[p]$ in $R_m = \order{K}/\pi_K^m \order{K}$. Define a homomorphism $\eta: \mor{\units{\Z[p]}}{\End[R_m]{D(N,R_m)}}$ by $\eta(x)(f) = [x]f$. By \cite{Deo17}, Lemma 10, we have $\im \eta \subset \set{T}(N, R_m)$. In particular, a $dc$-weak eigenform is also an eigenvector for the operators 
\[ t_\Lambda := \begin{cases}[1 + p]\indent p \geq 3 \mbox{ or } N \geq 3, \\ [1+4] \indent \mbox{ otherwise }.\end{cases}\] 

In \cite{Ru17}, we obtained the following theorem.

\begin{theorem}[\cite{Ru17}, Theorem 6.2]\label{comparedcweakI} Let $f \in S(N,R_m)$ be a $dc$-weak eigenform. Then $f$ is weak if and only if the eigenvalue of $f$ under the action of $t_\Lambda$ lies in $\Zmod{p}{\gamma(m)}$.
\end{theorem} 

The proof of \cref{comparedcweakI} in \cite{Ru17} relies on the results of \cite{Ka73} and \cite{Ka75}, and therefore holds for all $N \geq 1$ if $p \geq 5$ as well as for $N \geq 2$ if $p = 3$ and for $N \geq 3$ if $p = 2$ (which are the assumptions of \cite{Ka73} and \cite{Ka75}). However, with a little bit of extra work, we can see that the result actually holds for $N \geq 1$ and $p = 2$ or $3$. Note that when $N = p = 2$, we have $D(\Gamma_0(2), R_m) = D(1, R_m)$ by \cite{Go88}, Proposition I.3.9. So it is enough to consider the cases $N = 1$ and $p = 2$ or $3$. 

We will briefly review and explain the argument. There are two main ingredients used in the proof. First, it uses the action of $\units{\Z[p]}$ on divided congruences. As mentioned in \cref{section:modformdc}, this action still exists when $N = 1$ and $p = 2$ or $3$. If $f$ is a $dc$-weak eigenform, then $f$ is an eigenvector for the operators $[x]$ for all $x \in \units{\Z[p]}$. Arguing the same way as in the proof of \cite{Ru17}, Theorem 6.2, we construct from $f$ a rule $g$ defined on isomorphism classes of couples $(E/R, \omega)$ where $R$ is an $R_m$-algebra, $E$ is an elliptic curve over $R$, and $\omega$ is an invariant differential on $E$. The $q$-expansion of $g$ (i.e.\ its evaluation at the Tate curve) is the same as the $q$-expansion of $f$. Second, the proof uses the following result.

\begin{lemma} Let $g$ be a rule which to every $R_m$-algebra $R$ and every fibre-by-fibre ordinary elliptic curve $E/R$ together with an invariant differential $\omega$ assigns an element of $R$ depending only on the isomorphism class of $(E/R, \omega)$ and whose formation is compatible with base change. Then there exists a true modular form $h$ with coefficients in $R_m$ and whose $q$-expansion is equal to that of $g$. 
\end{lemma}

\begin{proof} If $p \geq 5$, this is Proposition 2.7.2 of \cite{Ka73}. Suppose $p = 2$ or $3$. The curve
\[ C_j : y^2 + xy = x^3 - 36(j-1728)^{-1}x - (j-1728)^{-1}   \]
is an elliptic curve with $j$-invariant $j$ defined over $R_m[j^{-1}]$ since
\[\frac{1}{(j-1728)} = \sum_{s \geq 0} 2^{6s}3^{3s}j^{-s-1}\]
and $p$ ($= 2$ or $3$) is nilpotent in $R_m$. Moreover, $C_j$ has fibre-by-fibre ordinary reduction as $j = 0$ is the only supersingular $j$-invariant in characteristics $2$ and $3$ (\cite{Si09}, \S V.4). Evaluating $g$ at $C_j$ with its canonical invariant differential gives us a polynomial $G$ in $j^{-1} = \Delta / E_4^3$ where $\Delta$ is the unique cuspform of level $1$ and weight $12$. The curve $C_j$ is in fact isomorphic to the Tate curve if $\inv{j}$ is $q$-expanded. Since the formation of $g$ commutes with base change, the $q$-expansion of $G$ obtained by expanding $\inv{j}$ is the $q$-expansion of $g$. Multiplying $G$ by $E_4^t$ where $t \gg 0$ and $E_4^t \equiv 1 \pmod{p^{\gamma(m)}}$, we obtain an isobaric polynomial\footnote{A polynomial in $E_4$ and $\Delta$ is isobaric if each of its monomial terms has the same weight.} in $E_4$ and $\Delta$. Therefore, the $q$-expansion of $g$ comes from the $q$-expansion of a true modular form in characteristic $0$.   
\end{proof}

The space $D(N, R_m)$ is usually much larger than $S(N, R_m)$, and a $dc$-weak eigenform in $D(N, R_m)$ does not necessarily lie in $S(N, R_m)$. In other words, $dc$-weak eigenforms in $D(N, R_m)$ might not be weak. We will later exhibit explicit examples of $dc$-weak eigenforms that are not weak. But first, we will completely determine which $dc$-weak eigenforms are weak.

Let $f \in D(N, R_m)$ be a $dc$-weak eigenform such that $t_\Lambda f = \lambda f$ for some $\lambda \in \Zmod{p}{\gamma(m)}$. By the Deligne-Serre lifting lemma (\cref{dslifting}), the image of $f$ in $D(N, \F[K])$ is strong, so in particular it is weak. Thus $\lambda \equiv 1 \pmod{p}$. When $p \geq 3$, $1 + p$ is a topological generator of $1 + p\Z[p]$, and therefore we have $t_\Lambda f = (1+p)^{\beta(f)} f$ for some $\beta(f) \in \Zmod{p}{\gamma(m)-1}$.

\begin{theorem}\label{dcweakisweak} Suppose $p \geq 3$. Let $f \in D(N, R_m)$ be a $dc$-weak eigenform such that the eigenvalue of $f$ under the action of $t_\Lambda$ lies in $\Zmod{p}{\gamma(m)}$. Then $f\in S_k(N, R_m)$ for some $k \equiv \beta(f) \pmod{p^{\gamma(m)-1}}.$ In particular, $f$ is a weak eigenform.  
\end{theorem}

\begin{proof} We proceed by induction on $m$. If $m = 1$, then $R_m = \F[K]$, and the statement follows from the Deligne-Serre lifting lemma (\cref{dslifting}). 

Suppose $m > 1$, and write $\beta = \beta(f)$. If we could show that $f \in S(N, R_m)$, then we would use \cref{comparedcweakI} to show that $f$ is weak and that $f \in S_{k'}(N, R_m)$ for some positive integer $k'$. In that case, after applying $[1+p]$ to $f$, we would end up with $(1+p)^{\beta} f = (1+p)^{k'} f$. This would give us $k' \equiv \beta \pmod{p^{\gamma(m)-1}}$, and we would be done. Thus, all we have to do is to show that $f \in S(N, R_m)$. 

The image of $f$ in $D(N, R_{m-1})$ is a $dc$-weak eigenform whose eigenvalue under the action of $t_\Lambda$ lies in $\Zmod{p}{\gamma(m-1)}$. Hence, by the inductive hypothesis, there exists a positive integer $k \equiv \beta \pmod{p^{\gamma(m-1)-1}}$ such that the image of $f$ in $D(N, R_{m-1})$ lies in $S_k(N, R_{m-1})$. Using the canonical surjective map $\surjmor{S_k(N, R_m)}{S_k(N, R_{m-1})}$ (\cite{Ru17}) and \cref{dcbasechange}, we can find $f_{m-1} \in S_k(N, R_{m})$ such that 
\[ f \equiv f_{m-1} \pmod{\pi_K^{m-1}D(N, R_m)}. \]
Write $k = \beta + tp^{\gamma(m-1)-1}$. Let
\[ h = \begin{cases} 4 \indent \mbox{if } p = 3, \\ p-1 \indent \mbox{if } p \geq 5 \end{cases}  \]
and consider the Eisenstein series $E_h$ of level $1$ and weight $h$. Then $p \nmid h$ and we can find a positive integer $a$ such that \[ha \equiv -t \pmod{p^{\gamma(m)-\gamma(m-1)}}.\] Furthermore, $E_h^{p^{\gamma(m-1)-1}}\equiv 1 \pmod{p^{\gamma(m-1)}}$ (see the top of this section). Hence
\[ f - f_{m-1}E_h^{ap^{\gamma(m-1)-1}} \equiv 0 \pmod{\pi_K^{m-1}D(N, R_m)}.  \]Thus there exists $g \in D(N, R_m)$ such that 
\[ f - f_{m-1}E_h^{ap^{\gamma(m-1)-1}} = \pi_K^{m-1} g.  \] 
The element $f_{m-1}E_h^{ap^{\gamma(m-1)-1}}$ has weight $k + hap^{\gamma(m-1)-1} \equiv \beta \pmod{p^{\gamma(m)-1}}$. Therefore, after applying $[1+p]$ to both sides, we get
\[ (1+p)^{\beta} f - (1+p)^{\beta}f_{m-1}E_h^{ap^{\gamma(m-1)-1}} = \pi_K^{m-1} [1+p] g   \]
and so $\pi_K^{m-1} \delta = 0$ 
where \[\delta := (1+p)^{\beta} g - [1+p]g.\]
Let $\bar\delta$ and $\bar{g}$ be, respectively, the images of $\delta$ and $g$ in $D(N, \F[K])$. Since $\pi_K^{m-1} a_n(\delta) = 0$ for all $n \geq 1$, we have $a_n(\delta) \in \pi_K R_m$ and therefore $a_n(\bar\delta) = 0$ for all $n \geq 1$. Consequently, $\bar\delta = 0$. On the other hand, \[\bar\delta = (1+p)^\beta \bar g - [1+p] \bar g = \bar g - [1+p]\bar g.\] 
Thus $[1+p]\bar{g} = \bar{g}$. By \cref{katzdcmodp}, this implies that $\bar{g} \in S(N, \F[K])$. Using the canonical surjective map $\surjmor{S(N, R_m)}{S(N, \F[K])}$ (\cite{Ru17}), we can lift $\bar{g}$ to an element $g' \in S(N, R_m)$. The elements $g$ and $g'$ satisfy
\[ g - g' \in \pi_KD(N, R_m)  \] and therefore
\[ \pi_K^{m-1}g = \pi_K^{m-1} g'.  \]
This means that $f \in S(N, R_m)$, which is what we needed.
\end{proof}

As a particular application of \cref{dcweakisweak}, we find that for $p \geq 3$ every $dc$-weak eigenform in $D(N, \Zmod{p}{m})$ is weak since the corresponding $t_\Lambda$ eigenvalue must lie in $\Zmod{p}{m}$. It is very telling that the proof of \cref{dcweakisweak} fails when $p = 2$. As we will see, this is because non-weak $dc$-weak eigenforms with coefficients in $\Zmod{2}{m}$ actually exist. The idea for the next proposition and the definition of the element $d$ come from Key Lemma 2.5 in \cite{Ka75}.
\begin{proposition}\label{dcweakmod4} Let $f \in D(\Zmod{4}{})$ be a $dc$-weak eigenform. Then either $f$ is weak or 
\[ f = f_0 + 2d\Delta \] 
for some $f_0 \in S(\Zmod{4}{})$ such that $f_0 \equiv \Delta \pmod{2D(\Zmod{4}{})}$ and $d$ is the image in $D(\Zmod{4}{})$ of
\[ \frac{E_4 - 1}{16}. \]
\end{proposition}
\begin{proof} By the Deligne-Serre lifting lemma and \cref{hatada2}, $f \equiv \Delta \pmod{2D(\Zmod{4}{})}$. So there exists $g \in D(\Zmod{4}{})$ such that $f = \Delta + 2g$. Let $\lambda \in \Zmod{4}{}$ such that $t_\Lambda f = \lambda f$. Then $\lambda \equiv 1 \pmod{2}$, which means that either $\lambda = 1$ or $\lambda = -1$. In any case, we have $t_\Lambda^2 f = f$. Since $t_\Lambda \Delta = \Delta$, we get $2t_\Lambda^2 g = 2g$.

Let $\bar{g}$ and $\bar{d}$, respectively, be the images of $g$ and $d$ in $D(\F[2])$. Note that $t_\Lambda^2 \bar{g} = \bar{g}$. Let $A$ and $V$, respectively, be the $\F[2]$-subalgebras of $\F[2]\dbr{q}$ generated by $S(\F[2])$ and $D(\F[2])$. Then $A \subset V$, and the action of $t_\Lambda$ extends to $V$. By \cref{katzdcmodp}, $V^{t_\Lambda} = A$. 

Let $B = V^{t_\Lambda^2}$. Then $A \subset B$, and $B$ is a finite \'etale $A$-algebra of rank $2$, which is Galois with group $\Zmod{2}{}$ (see  \cite{Ka75a}, \S X and \cite{Ka75}, (2.4)). We can easily check that 
\[ \bar{d} \in B, \]
\[t_\Lambda \bar{d} = \bar{d}+ 1,   \] 
\[\bar{d} - \bar{d}^2 = \bar{\Delta} \in A    \]
where $\bar{\Delta}$ is the image of $\Delta$ in $A$. Using Artin-Schreier theory (c.f.\ \cite{Ka75}, (2.4)), we find that $B = A[\bar{d}]$. Since $\bar{g} \in B$, we find that either $f \in S(\Zmod{4}{})$ (which corresponds to $\lambda = 1$) or $f = f_0 + 2df_1$ for some $f_0, f_1 \in S(\Zmod{4}{})$ (which corresponds to $\lambda = -1$). If it is the latter case, we apply $t_\Lambda$ to $f$ and get
\[ t_\Lambda f - f = f_0 + 2(d+1)f_1 - f_0 - 2df_1 = 2f_1, \]
hence $2f_1 = 2\Delta$ and therefore $f = f_0 + 2d\Delta$.     
\end{proof}


\section{$dc$-weak eigenforms over $\Zmod{4}{}$}\label{section:wefmod4}
In this section, we will calculate all $dc$-weak eigenforms with coefficients in $\Zmod{4}{}$. Recall that the graded algebra $M(\Z)$ of modular forms of level $1$ with coefficients in $\Z$ has the presentation (\cite{De75})
\[M(\Z) = \Z[][E_4, E_6, \Delta] / ( E_4^3 - E_6^2 - 1728\Delta).   \]
As $E_4 \equiv E_6 \equiv 1 \pmod{4}$, every element $f \in S(\Zmod{4}{})$ can be written uniquely as
\[ f = F \in \left(\Zmod{4}{}\right)[\Delta]  \]
It makes sense to write $\deg f$ for the degree of $f$ as a polynomial in $\Delta$. Note that knowing $\deg f$ is equivalent to knowing the weight in which $f$ occurs in $S(\Zmod{9}{})$. 

Let $\mathcal{T} = \{T_3, T_5, U\}$. We can reduce the computation of weak eigenforms to a finite process due to the following result.
\begin{proposition}\label{weightboundmod4} Let $f \in S(\Zmod{4}{})$ and suppose that
\[ \max\{\deg T f : T \in \sh{T}  \} \leq 1.   \] Then $\deg f \leq 5$. 
\end{proposition}
\begin{proof} This follows from the results of \cite{KRW16} (see Proposition 15 and the proof of Theorem 13 of that paper), which in turn rely on Nicolas-Serre theory (\cite{NS12}). 
\end{proof}

\begin{lemma}\label{heckeond}
We have $T(2d\Delta) = 2\Delta$ for all $T \in \sh{T}.$
\end{lemma}

\begin{proof} We check this by explicit computation using the fact that $E_4 \Delta$ is an eigenform in characteristic $0$.
\end{proof}

\begin{proposition}\label{weightboundefmod4}The $dc$-weak eigenforms in $f \in D(\Zmod{4}{})$ are of the form
\[  f = f_0, \indent f_0 \in \left(\Zmod{4}{}\right)[\Delta], \indent \deg f_0 \leq 5, \]
or
\[   f = f_0 + 2d\Delta, \indent f_0 \in \left(\Zmod{4}{}\right)[\Delta],  \indent \deg f_0 \leq 5.  \]

\end{proposition}
\begin{proof} Let $f \in D(\Zmod{4}{})$ be a $dc$-weak eigenform. By \cref{dcweakmod4}, we can write
\[ f = f_0 \indent \mbox{ or }\indent f = f_0 + 2d\Delta  \] where $f_0 \in \left(\Zmod{4}{}\right)[\Delta]$ is such that $f_0 \equiv \Delta \pmod{2D(\Zmod{4}{})}$. In any case, we have $a_2(f) \equiv a_3(f) \equiv a_5(f) \equiv 0 \pmod{2}$. Thus for each $T \in \sh{T}$ there exists $\lambda_T \in 2\Zmod{4}{}$ such that
$Tf = \lambda_T \Delta$.
On the other hand, we have
\[ Tf = Tf_0 \indent \mbox{ or } \indent Tf = Tf_0 + 2\Delta \]
for each $T \in \sh{T}$ by \cref{heckeond}. This gives us the bound
\[ \max\{\deg T f_0 : T \in \sh{T} \} \leq 1.   \]
By \cref{weightboundmod4}, we get that $\deg f_0 \leq 5$. 
\end{proof}

\begin{proposition}\label{allwefmod4}The set map
\[ \{\mbox{$dc$-weak eigenforms in }D(\Zmod{4}{})\} \rightarrow \left(2\Zmod{4}{} \right)^4,   \] 
\[ f \mapsto \left(\lambda-1, a_2(f), a_3(f), a_5(f) \right),\]
where $\lambda$ is such that $[1+4]f = \lambda f$, is bijective. 
\end{proposition}
\begin{proof} Using \cref{weightboundefmod4}, a brute force search in Sage gives us all the $dc$-weak eigenforms of level $1$ with coefficients in $\Zmod{4}{}$ and allows us to prove the proposition. 
\end{proof}
\begin{remark} Given a particular element $f \in D(\Zmod{p}{m})$, we can check that it is an eigenform by checking that $T_nf = a_n(f)$ for all $n$ up to a certain bound. This bound can be derived from the classical Sturm bound (see \cite{Ki08}, Theorem 3.13) for modular forms. 
\end{remark}

\cref{allwefmod4} gives us several $dc$-weak eigenforms of level $1$ which are not weak, one example being $\Delta + 2d\Delta \in D(\Zmod{4}{})$. These are the first explicitly known examples of non-weak $dc$-weak eigenforms. 


\section{$dc$-weak eigenforms over $\Zmod{9}{}$}\label{section:wefmod9}
In this section, we will calculate all $dc$-weak eigenforms with coefficients in $\Zmod{9}{}$ using the same sort of argument as was used in \cref{section:wefmod4}. By \cref{dcweakisweak}, all such $dc$-weak eigenforms are weak. 

As $E_4^3 \equiv E_6 \equiv 1 \pmod{9}$, every element $f \in S(\Zmod{9}{})$ can be written uniquely as
\[ f = F_0 + E_4 F_1 + E_4^2 F_2  \]
where $F_0, F_1, F_2 \in \left(\Zmod{9}{} \right)[\Delta]$. We write $\deg f := \max\{\deg F_i : i = 0,1,2\}$ where $\deg F_i$ is the degree of $F_i$ as a polynomial in $\Delta$. In particular, a weak eigenform in $S(\Zmod{9}{})$ must be of the form
\[ f = E_4^i F \]
for some $F \in \left(\Zmod{9}{}\right)[\Delta]$ and $i \in {0,1,2}$. Note that knowing $\deg f$ is equivalent to knowing the weight in which $f$ occurs in $S(\Zmod{9}{})$. 

Similarly, every element $\bar{g} \in S(\F[3])$ can be written as a polynomial in $\Delta$ with coefficients in $\F[3]$, and we can define $\deg \bar{g}$ to be the degree in $\Delta$ of the polynomial representing $\bar{g}$. 

Let $\mathcal{T} = \{T_2, 1 + T_7, U\}$.
\begin{proposition}\label{weightboundmod9} Let $\bar{g} \in S(\F[3])$. Suppose that
\[ \max\{\deg T\bar{g} : T \in \sh{T} \} \leq 1.   \]
Then $\deg \bar{g} \leq 10$. 
\end{proposition}

\begin{proof}In terms of the concepts described in \cite{Ru17}, this is a bound on the nilpotence filtration of $\bar{g}$, and we can turn it into a bound on the weight filtration (i.e.\ $\deg \bar{g}$). In the following, we describe how to do that explicitly. The argument is similar to the one used in \cite{KRW16}.

Since we are in characteristic $3$, we can split $\bar{g}$ into two parts 
\[ \bar{g} = \sum_i c_i \Delta^i = \bar{g}_o + \bar{g}_e^3  \]
where $\bar{g}_o = \sum_{i \not \equiv 0 \pmod{3}} c_i \Delta^i$. Then 
\[ U(\bar{g}_o) = 0,\indent U(\bar{g}_e^3) = \bar{g}_e,   \]
and so $\deg \bar{g}_e \leq 1$. Additionally, the operators $T_2$ and $1+T_7$ commute (modulo $3$) with the operator $\Delta \mapsto \Delta^3$. Because $T_2(\Delta) \equiv (1+T_7)(\Delta) \equiv 0 \pmod{3}$, we get $T_2(\bar{g}_e^3) = (1+T_7)(\bar{g}_e^3) = 0$.

To deal with $\bar{g}_o$, we use results from \cite{BK15}. Let $S_o$ be the subspace of $S(\F[3])$ consisting of those polynomials $\sum_i d_i \Delta^i $ such that $d_i = 0$ whenever $i \equiv 0 \pmod{3}$. Then $\bar{g}_o \in S_o$. By Corollary 25 of \cite{BK15}, there exists a unique basis $\{m(a,b)\}_{a,b \in \Z[\geq 0]}$ for $S_o$ such that
\begin{enumerate}[(i)]
\item $m(0,0) = \Delta$,
\item $T_2 m(a,b) = m(a-1, b)$ if $a \geq 1$ and $T_2 m(0,b)= 0$,
\item $(1+T_7)m(a,b) = m(a,b-1)$ if $b \geq 1$ and $(1+T_7)m(a,0) = 0$,
\item $a_1(m(a,b)) = 0$ unless $(a,b) = (0,0)$. 
\end{enumerate}
The first few elements of this basis are given in Example 26 of \cite{BK15}. The ones we will need are
\[ m(0,1) = \Delta^7 + 2\Delta^{10},  \]
\[ m(1,0) = \Delta^2. \]

Write $\bar{g}_o = \sum_{a,b \in \Z[\geq 0]} \alpha_{a,b} m(a,b)$. First apply $T_2$. The properties of the basis $m(a,b)$ then tell us that
\[\deg T_2 \bar{g}_o \leq 1 \Rightarrow \{ \alpha_{a,b}: a \geq 1 \mbox{ and } (a,b)\not = (1,0)  \} = \{0 \}.   \]
Thus $\bar{g}_o = \alpha_{1,0} m(1,0) + \sum_{b\geq 0} \alpha_{0,b}m(0, b)$. Now apply $1 + T_7$. Then, similarly, we have 
\[ \deg(1+T_7)(\bar{g}_o)\leq 1 \Rightarrow \{ \alpha_{0, b} : b\geq 2\} = \{0\}.  \]
We deduce that $\bar{g}_o = \alpha_{0,0}\Delta + \alpha_{1,0}m(1,0) + \alpha_{0,1}m(0,1)$ and therefore $\deg \bar{g}_o \leq 10$. Combining all the above bounds together, we get $\deg \bar{g} \leq 10$. 
\end{proof}

\begin{remark} Medvedovsky has calculated\footnote{\url{https://www.math.brown.edu/~medved/Mathdata/ModFormsMod3/mab3upto17.txt}} all the elements of the basis $m(a,b)$ for $a+b \leq 17$. 
\end{remark}

\begin{proposition}\label{weightboundefmod9} Let $f \in S(\Zmod{9}{})$ be a weak eigenform. Then $\deg f \leq 10$. 
\end{proposition}
\begin{proof} The image of $f$ in $S(\F[3])$ is strong by the Deligne-Serre lifting lemma (\cref{dslifting}), hence $f \equiv \Delta \pmod{3}$ by \cref{hatada3}. Since $f$ is weak, we may write
\[ f = E_4^i \Delta + 3g  \]
where $g \in S(\Zmod{9}{})$ and $i \in \{0,1,2\}$. Let $\bar{g}$ denote the image of $g$ in $S(\F[3])$. We only need to bound $\deg \bar{g}$. By \cref{hatada3}, we have $a_2(f) \equiv a_3(f) \equiv 1 + a_7(f) \equiv 0 \pmod{3}$. 
So for each $T \in \sh{T} = \{T_2, 1+T_7, U\}$ there exists $\lambda_T \in 3 \Zmod{9}{}$ such that
\[ T f = \lambda_T f = \lambda_T E_4^i \Delta = \lambda_T \Delta.  \]
On the other hand, the forms $E_4^i \Delta$ are eigenforms in characteristic $0$ for $i \in \{0,1,2\}$, so for each $T \in \sh{T}$ there exists $\lambda_T' \in 3\Zmod{9}{}$ (again by \cref{hatada3}) such that 
\[ T f = \lambda_T' E_4^i\Delta + 3 Tg = \lambda_T' \Delta + 3 Tg. \]
Therefore $\max\{ \deg T\bar{g} : T \in \sh{T} \} \leq 1$ and by \cref{weightboundmod9} we get $\deg \bar{g} \leq 10$. 
\end{proof}

\begin{proposition}\label{allwefmod9}All $dc$-weak eigenforms in $D(\Zmod{9}{})$ occur in the spaces $S_{120+4i}(\Zmod{9}{})$ with $i \in \{0,1,2\}$. The set map
\[ \{\mbox{$dc$-weak eigenforms in }D(\Zmod{9}{})\} \rightarrow \left(3\Zmod{9}{} \right)^4,   \] 
\[ f \in S(\Zmod{9}{}) \mapsto \left(\lambda - 1, a_2(f), a_3(f), 1+a_7(f) \right),\]
where $\lambda \in \Z$ is such that $[1+3]f = \lambda f$, is bijective. 
\end{proposition}
\begin{proof} Using the fact that $dc$-weak eigenforms with coefficients in $\Zmod{9}{}$ are weak together with the weight bounds obtained in \cref{weightboundefmod9}, a brute force search in Sage for all weak eigenforms in the spaces $S_{120+4i}(\Zmod{9}{})$ for $i \in \{0,1,2\}$ gives us all $dc$-weak eigenforms of level $1$ with coefficients in $\Zmod{9}{}$ and allows us to prove the proposition. There are $81$ in all. Up to twist there are only $27$ and they all occur in $S_{120}$ (note that $\lambda$ and $i$ completely determine each other). 

\end{proof}


\section{Generators of the Hecke algebras}\label{section:heckegen}
Let $(A, \idl{m})$ be a local $\Z[p]$-algebra such that $pA \not = 0$ and $A/\idl{m} = \F[p]$. Let $\mor[\varphi]{A}{\Zmod{p}{2}}$ be a morphism of local $\Z[p]$-algebras. Then we necessarily have $\varphi(\idl{m}^2) = 0$ as $\varphi(\idl{m}) \subset p\Zmod{p}{2}$. This means that the restriction of $\varphi$ to $\idl{m}$ induces an $\F[p]$-vector space map
\[ \tilde{\varphi}: \mor{\idl{m}/\idl{m}^2}{p\Zmod{p}{2}}  \]
where $p\Zmod{p}{2}$ is regarded as a $1$-dimensional $\F[p]$-vector space with basis $p$. So we have a set map
\[ \mor[r]{\Hom{A}{\Zmod{p}{2}}}{\Hom[\set{F}_p]{\idl{m}/\idl{m}^2}{p\Zmod{p}{2}}}.    \]
The following lemma is well-known, but we include a proof of it as we could not locate it in the literature.  
\begin{lemma}\label{modp2homs} The map $r$ is injective. Its image consists precisely of the $\F[p]$-linear maps $\tilde{\varphi}: \idl{m}/\idl{m}^2 \rightarrow p\Zmod{p}{2}$ such that $\tilde{\varphi}(p) = p$. 
\end{lemma}
\begin{proof} Without loss of generality, we may assume that $\idl{m}^2 = 0$. So the image of $\Z[p]$ in $A$ is $\Zmod{p}{2}$. Since the residue field of $A$ is $\F[p]$, every $x \in A$ can be written as $x = x_0 + x_\idl{m}$ where $x_0 \in \Zmod{p}{2}$ and $x_\idl{m}\in \idl{m}$. 

Let $\varphi, \psi \in \Hom{A}{\Zmod{p}{2}}$ such that $\varphi|_{\idl{m}} = \psi|_{\idl{m}}$. Take $x = x_0 + x_\idl{m} \in A$. Then $\varphi(x_0) = \psi(x_0)$ as $\varphi$ and $\psi$ are $\Z[p]$-algebra homomorphisms, and $\varphi(x_\idl{m}) = \psi(x_\idl{m})$ by assumption, hence $\varphi(x) = \psi(x)$. This shows that $r$ is injective.

On the other hand, let $\tilde{\varphi}: \mor{\idl{m}}{p\Zmod{p}{2}}$ be an $\F[p]$-linear map such that $\tilde{\varphi}(p) = p$. We will show how to extend $\tilde{\varphi}$ to an algebra homomorphism. Let $x = x_0 + x_\idl{m} \in A$, and set $\varphi(x) := x_0 + \tilde{\varphi}(x_\idl{m})$. We explain why this assignment is well-defined. If $x = x_0' + x_\idl{m}'$ is another decomposition of $x$ with $x_0' \in \Zmod{p}{2}$ and $x_\idl{m}' \in \idl{m}$, then $x_0 - x_0' \in \Zmod{p}{2} \bigcap \idl{m} = p\Zmod{p}{2}$. So $x_0-x_0' = cp$ for some $c \in \F[p]$, and \[\tilde{\varphi}(x_0-x_0') = c\tilde{\varphi}(p) = cp = x_0-x_0'.\] This means that
\[ x_0 + \tilde{\varphi}(x_\idl{m}) - x_0' - \tilde{\varphi}(x_\idl{m}') = \tilde{\varphi}(0) = 0.  \]

It remains to check that the map $\varphi$ defined in this manner is multiplicative. First note that if $c \in \Zmod{p}{2} \subset A$ and $x \in \idl{m}$, then $\tilde{\varphi}(cx) = c\tilde{\varphi}(x)$ as these two expressions only depend on the residual image of $c$ in $\F[p]$. 

Now let $x = x_0 + x_\idl{m}$ and $y = y_0 + y_\idl{m}$ be elements of $A$. Then
\[ xy = x_0y_0 + x_0y_m + x_m y_0  \]
and
\[\varphi(xy) = x_0 y_0 + \tilde{\varphi}(x_0 y_m + x_my_0).  \]
But
\[\varphi(x)\varphi(y) = x_0 y_0 + x_0\tilde{\varphi}(y_m) + y_0\tilde{\varphi}(x_m) = x_0y_0 + \tilde{\varphi}(x_0y_m + x_my_0) = \varphi(xy).    \] 
\end{proof}

We now apply this to the Hecke algebra in characteristic $0$. When $p = 2$ or $3$, there is only one mod $p$ eigenform in level $1$, so the Hecke algebras $\set{T}(2, \Z[2])$ and $\set{T}(3, \Z[3])$ are local rings. 

\begin{theorem}\label{heckealggen}\leavevmode
\begin{enumerate}[(i)]
\item The Hecke algebra $\set{T}(2,\Z[2])$ is generated as a $\Z[2]$-algebra by the operators $U$, $T_3$, and $T_5$, and $t_\Lambda := [1+4]$.
\item The Hecke algebra $\set{T}(3,\Z[3])$ is generated  as a $\Z[3]$-algebra  by the operators $U$, $T_2$, $1+T_7$, and $t_\Lambda := [1+3]$.
\end{enumerate}

\end{theorem} 
\begin{proof} Let $\idl{m}$ be the maximal ideal of $\set{T}(p,\Z[p])$. By \cref{allwefmod4} and \cref{allwefmod9} we see that $\Hom{\set{T}}{\Zmod{p}{2}}$ is a finite set of $p^4$ elements. Using \cref{modp2homs}, we deduce that $\idl{m}/\idl{m}^2$ is finite dimensional over $\F[p]$ and that $p^{\dim \idl{m}/\idl{m}^2 - 1} = p^4$. Thus $\dim \idl{m}/\idl{m}^2 = 5$. Let 
\[
\mathcal{S} = \begin{cases}\{U, T_3, T_5, t_\Lambda -1 \}\indent\indent\mbox{  if } p=2, \\
\{U, T_2, 1+T_7, t_\Lambda -1 \} \indent \mbox{if } p = 3.   \end{cases}
\]
\cref{allwefmod4} for $p = 2$ and \cref{allwefmod9} for $p = 3$ show that for each $T \in \mathcal{S}$ there exists an $\F[p]$-linear map $\delta_T:\mor{\idl{m}/\idl{m}^2}{p\Zmod{p}{2}}$ such that
\[ \delta_T(p) = \delta_T(T) = p \mbox{  and  } \delta_T(T') = 0 \mbox{  for all } T' \in \mathcal{S}\setminus\{T\}.  \]
The existence of these maps shows that the images of $\mathcal{S}$ in $\idl{m}/\idl{m}^2$ are $\F[p]$-linearly independent. 

Suppose $\alpha\in\F[p]$ and $\{\alpha_T\}_{T\in\mathcal{S}}\subset \F[p]$ such that
\[ \alpha(p + \idl{m}^2) + \sum_{T\in\mathcal{S}}\alpha_T (T+\idl{m}^2) = 0 \tag{$\ast$}\label{linindep}.  \]
Applying each $\delta_T$ to \cref{linindep} in turn, we find that $\alpha = -\alpha_T$ for all $T\in\mathcal{S}$. 

For $p = 2$, put $T = T_3$ and $T' = T_5$, and for $p = 3$ put $T = T_2$ and $T' = 1+T_7$. By \cref{allwefmod4} for $p = 2$ and \cref{allwefmod9} for $p = 3$, there exists an $\F[p]$-linear map $\delta: \idl{m}/\idl{m}^2 \rightarrow p\Zmod{p}{2}$ such that $\delta(p) = \delta(T) = \delta(T') = p$ and 
\[ \delta(T'') = 0 \mbox{   for all } T'' \in \mathcal{S}\setminus \{T, T'\}. \]
Applying $\delta$ to \cref{linindep}, we find that $\alpha = 0$. Thus the image of $\{p\} \cup  \mathcal{S}$ in $\idl{m}/\idl{m}^2$ is $\F[p]$-linearly independent and therefore constitutes an $\F[p]$-basis of $\idl{m}/\idl{m}^2$. By Nakayama's lemma, $\mathcal{S}$ generates $\set{T}(p, \Z[p])$ as a $\Z[p]$-algebra. 


\end{proof}
\begin{remark} In fact, the algebras $\set{T}^{pf}(\Z[2])$ and $\set{T}^{pf}(\Z[3])$ are power series rings in three variables over (respectively) $\Z[2]$ and $\Z[3]$. This can be shown using the deformation theory of pseudo-representations (\cite{Ch14}, \cite{Be12}) together with the Gouv\^ea-Mazur infinite fern argument (\cite{GM98}, \cite{Em11b}). 
\end{remark}

\section{Finiteness of strong eigenforms modulo $128\Zc[2]$}\label{section:finiteness2}
In this section, we will prove the following theorem. 
\begin{theorem}\label{finiteness} If $f$ is a level $1$ eigenform of weight $k$ and $\ell$ is an odd prime then $a_\ell(f) \equiv 1+\ell^{k-1} \pmod{128 \Zc[2]}$. Consequently, there are only finitely many congruence classes $\pmod{128\Zc[2]}$ of eigenforms of level $1$.
\end{theorem}

First we will prove this for eigenforms of low weight.
\begin{proposition}\label{kolbergtype} If $f$ is an eigenform of level $1$ and weight $k$, then for every odd integer $n \geq 1$ we have $a_n(f) \equiv \sigma_{k-1}(n) \pmod{128\Zc}$. 
\end{proposition}
\begin{proof} Let $k$ be as in the statement. For each such $k$, let the space $M_k(\Gamma_0(2), \Z)$ of modular forms on $\Gamma_0(2)$ with integral coefficients. We check on Sage that each of these spaces has a Victor Miller basis, i.e.\ an integral basis $\{b_0, \ldots, b_{d(k)}\}$ where 
\[d(k) = \dim M_k(\Gamma_0(2), \C) = \dim M_k(\Gamma_0(2), \Z)\otimes \C\]
such that
\[ b_i = q^i + O(q^{d(k)+1}). \]
Let $f$ be an eigenform of level $1$ and weight $k$, and let $G_k$ be the Eisenstein series
\[ G_k = -\frac{B_k}{2k}E_k = -\frac{B_k}{2k} + \sum_{n \geq 1}\sigma_{k-1}(n)q^n. \]
Put $g = f - G_k$. If $V$ is the operator
\[ V\left(\sum a_n q^n \right) = \sum a_n q^{2n},  \]
then the form $h = g - V(U(g))$ is a modular form on $\Gamma_0(2)$ with coefficients in $\Zc$ (the constant term of $G_k$, which is not $2$-integral, is cancelled out). Moreover, we have
\[ a_n(g) = \sum_{\substack{n\geq 1\\n \equiv 1 \pmod{2}}}a_n(f - G_k)q^n.  \]
The statement that 
\[ a_n(f) \equiv \sigma_{k-1}(n) \pmod{128\Zc} \]
is then equivalent to the statement that $h \equiv 0 \pmod{128\Zc}$. Since we have a Victor Miller basis for each space $M_k(\Gamma_0(2), \Zc)$, this in turn is equivalent to the statement that
\[ a_n(h) \equiv 0 \pmod{128 \Zc}\indent \forall 1\leq n \leq d(k).  \]
Thus it is enough to check that
\[ \tag{$\dagger$}\label{vmcheck} a_n(f) \equiv \sigma_{k-1}(n) \pmod{128\Zc} \indent \forall 1 \leq n \leq d(k), n \equiv 1 \pmod{2}.  \]

We will check this on Sage. When $k \leq 12$ or $k = 14$, there are no normalised cuspidal eigenforms. For other values of $k$, we check that there is a single Galois orbit\footnote{Recall that Maeda's conjecture predicts that there is a single Galois orbit for all $k$.} of eigenforms of level $1$ and weight $k$ and we pick a representative $f$ of the Galois orbit. Clearly, if 
\[ a_n(f) \equiv \sigma_{k-1}(n) \pmod{128\Zc} \indent \forall 1 \leq n \leq d(k), n \equiv 1 \pmod{2} \]
and $\sigma \in \Gal(\Qc/\Q)$, then
\[ a_n(\sigma(f)) \equiv \sigma_{k-1}(n) \pmod{128\Zc} \indent \forall 1 \leq n \leq d(k), n \equiv 1 \pmod{2}.  \]
This allows us to verify the statement (\ref{vmcheck}) on Sage for all $12 \leq k \leq 46$ with $k \not = 14$. 
\end{proof}
\begin{remark} For $k = 12$, the result in \cref{kolbergtype} follows from a theorem of Bambah and Chowla (\cite{BC47}). See also Kolberg's strengthening of this result for $k = 12$ in \cite{Ko62} (c.f.\ \cite{Sw73}).
\end{remark}

We will also use the following proposition.

\begin{proposition}\label{actiontmod32} If $f$ is a level $1$ eigenform of weight $k$, then 
\[ a_3(f) \equiv 1 + 3^{k-1} \pmod{128\Zc[2]}  \] and
\[ a_5(f) \equiv 1 + 5^{k-1} \pmod{128\Zc[2]}. \]
\end{proposition}
\begin{proof} See \cref{section:algorithm}.
\end{proof}

\begin{proof}[Proof of \cref{finiteness}] By Corollary 11 of \cite{KRW16} (which uses results of Coleman and Wan), there are only finitely many congruence classes modulo $128$ of eigenforms $f$ of level $1$ such that $a_2(f) \not \equiv 0 \pmod{128\Zc[2]}$. Thus we only need to consider eigenforms $f$ of level $1$ such that $a_2(f) \equiv 0 \pmod{128 \Zc[2]}$. 

Let $\sh{O}_f$ be a finite extension of $\Z[2]$ containing the coefficients of $f$, and let $R = \sh{O}_f/128\sh{O}_f$. Let $\bar{f}$ be the image of $f$ in $D(R)$. In \cref{section:swdc} we saw that $\set{T}(R) = \set{T}(\Z[2])\otimes_{\Z[2]} R$ and that $\bar{f}$ corresponds to an $R$-algebra homomorphism $\varphi: \mor{\set{T}(R)}{R}$. By \cref{heckealggen}, $\set{T}(2, R)$ is generated by $T_3, T_5, U$, and $t_\Lambda$. But $\set{T}(2, R) = \set{T}(R)$ by \cref{comparehecke}. Hence by \cref{heckealggen},  $a_n(f) = \varphi(T_n)$ is a polynomial with coefficients in $\Zmod{128}{}$ in $a_2(\bar{f}) = \varphi(U)$, $a_3(\bar{f}) = \varphi(T_3)$, $a_5(\bar{f}) = \varphi(T_5)$, and the eigenvalue $\lambda$ corresponding to $t_\Lambda = [1+4]$.

We have assumed $a_2(\bar{f}) = 0$, and $\lambda$ clearly lies in $\Zmod{128}{}$. Furthermore we have $\{a_3(\bar{f}), a_5(\bar{f})\} \subset \Zmod{128}{}$ by \cref{actiontmod32}, and these numbers are determined by $k \pmod{32}$. Thus $f \in S(\Zmod{128}{})$ and has to be congruent away from $2$ to an eigenform of level $1$ and weight $k$ with $k \not = 14$ and $12 \leq k \leq 46$. Now \cref{finiteness} follows from this using \cref{kolbergtype}. 
\end{proof}

\section{Finiteness of strong eigenforms modulo $9\Zc[3]$}\label{section:finiteness3}
We conjecture that the following is true.
\begin{conjecture}\label{a2conj} Let $f$ be an eigenform of level $1$ and weight $w + 2$. Then 
\[ a_2(f) \equiv \begin{cases} 3\mbox{ or } 6\pmod{9\Zc[3]}\indent \mbox{if } w \equiv 0 \pmod{6} \\ 3 \mbox{ or } 6\pmod{9\Zc[3]}\indent \mbox{if } w \equiv 4 \pmod{6}\\ 0 \pmod{9\Zc[3]}\indent\indent\indent \mbox{if } w \equiv 2 \pmod{6}.\end{cases} \]
\end{conjecture}
Just as we did in \cref{section:finiteness2}, we will prove the following. 
\begin{theorem} \label{finitenessa2} \cref{a2conj} implies that there are only finitely many congruence classes $\pmod{9\Z[3]}$ of eigenforms of level $1$.
\end{theorem}
\cref{finitenessa2} will follow from the following proposition.
\begin{proposition}\label{actiont7}If $f$ is a level $1$ eigenform of weight $k = w + 2$, then 
\[ a_7(f) \equiv \begin{cases}5 \pmod{9\Zc[3]} \indent\mbox{if } k \equiv 0 \pmod{6},\\ 8 \pmod{9\Zc[3]}\indent  \mbox{if } k \equiv 2 \pmod{6},\\2 \pmod{9\Zc[3]}\indent  \mbox{if } k \equiv 4 \pmod{6}. \end{cases}  \]
\end{proposition}
\begin{proof} See \cref{section:algorithm}.
\end{proof}

\begin{proof}[Proof of \cref{finitenessa2}] Assume \cref{a2conj}. By Congruence (6) of \cite{Ha79}, we have $a_3(f)\equiv 0 \pmod{9\Zc[3]}$.

Let $\sh{O}_f$ be a finite extension of $\Z[3]$ containing the coefficients of $f$, and let $R = \sh{O}_f/9\sh{O}_f$. Let $\bar{f}$ be the image of $f$ in $D(R)$. In \cref{section:swdc} we saw that $\set{T}(R) = \set{T}(\Z[3])\otimes_{\Z[3]} R$ and that $\bar{f}$ corresponds to an $R$-algebra homomorphism $\varphi: \mor{\set{T}(R)}{R}$. By \cref{heckealggen}, $\set{T}(3, R)$ is generated by $T_2, 1+T_7, U$, and $t_\Lambda$. But $\set{T}(3, R) = \set{T}(R)$ by \cref{comparehecke}. Thus by \cref{heckealggen}, $a_n(f) = \varphi(T_n)$ is a polynomial with coefficients in $\Zmod{9}{}$ in $a_3(\bar{f}) = \varphi(U)$, $a_2(\bar{f}) = \varphi(T_2)$, $a_7(\bar{f}) = \varphi(T_7)$, and the eigenvalue $\lambda$ corresponding to $t_\Lambda = [1+3]$.

The eigenvalue $\lambda$ clearly lies in $\Zmod{9}{}$. By \cref{actiont7} and the assumption, $\{a_2(\bar{f}), a_7(\bar{f})\} \subset \Zmod{9}{}$. Thus $f \in S(\Zmod{9}{})$, and there are only finitely such eigenforms. 
\end{proof}

\section{Proving congruences for specific eigenvalues}\label{section:algorithm}
We will reduce the proof \cref{actiontmod32} and \cref{actiont7} to the verification of a finite number of polynomial identities. For \cref{actiontmod32}  we need $2048$ identities, and for \cref{actiont7} we need $9$ identities. In this section, we will describe an algorithm that can discover and verify the required identities. The algorithm relies on the theory of modular symbols, which has a simple presentation in the level $1$ case. The main reference for this theory is \cite{Me94}. 

Let $w \geq 0$ be an integer, and let $\Z[][X,Y]_w$ be the $\Z$-module of homogeneous polynomials in $X$ and $Y$ of degree $w$. Then $\Z[][X,Y]_w$ has a standard basis consisting of the monomials $\{X^i Y^{w-i}\}_{i = 0}^w$. If $\gamma = \mat{a}{b}{c}{d} \in M_2(\Z)$ is a $2 \times 2$ integral matrix and $P \in \Z[][X,Y]_w$, we let $\gamma$ act on $P$ on the right by
\[ P[\gamma] := P(aX + bY, cX + dY).  \]
This action extends to a right action of the monoid-ring $\Z[][M_2(\Z)]$ on $\Z[][X,Y]_w$.

Let
\[ \sigma := \mat{0}{-1}{1}{0}, \tau := \mat{0}{-1}{1}{-1}.  \]
Note that $\sigma$ has order $2$ and $\tau$ has order $3$. To simplify exposition, we will introduce the following notation. For $\gamma \in M_2(\Z)$ and $P\in M_2(\Z)$, write
\[ \cyc{P}{\gamma}{\tau} := P[\gamma + \gamma\tau + \gamma\tau^2],  \]
\[ \cyc{P}{\gamma}{\sigma} := P[\gamma + \gamma\sigma].  \]
Note that $\cyc{P}{\gamma}{\tau} = \cyc{P[\gamma]}{1}{\tau}$ and $\cyc{P}{\gamma}{\sigma} = \cyc{P[\gamma]}{1}{\sigma}$. 

The space of integral modular symbols of level $1$ and weight $w + 2$, denoted by $\set{M}_{w+2}$, is the quotient of $\Z[][X,Y]_w$ by the subgroup by
\[ \{\cyc{P}{1}{\sigma}, \cyc{P}{1}{\tau} : P \in \Z[][X,Y]_w\}  \]
and by any torsion. We denote by $\set{S}_{w+2}$ the subgroup of $\set{M}_{w+2}$ generated by the image of 
\[ \{X^i Y^{w-i}\}_{i=1}^{w-1} \bigcup \{X^w - Y^w\}. \]
We define an involution $\iota^\ast$ on $\set{M}_{w+2}$ by
\[ (\iota^\ast P)(X,Y) := -P(Y, X),  \]
and denote by $\set{S}_{w+2}^+$ the subspace of $\set{S}_{w+2}$ consisting of all elements fixed by $\iota^\ast$. 

Now we will describe the Hecke action on modular symbols. For each $n \geq 1$, we introduce the set $\mathcal{HM}_n$ of Heilbronn-Merel matrices of determinant $n$. These are given by
\[ \mathcal{HM}_n := \left\{\mat{a}{b}{c}{d} \in M_2(\Z) : a > b \geq 0, d > c \geq 0, ad-bc = n  \right\}.  \]

Merel (\cite{Me94}) showed that the Heilbronn-Merel matrices of determinant $n$ can be used to define an operator $T_n$ on modular symbols, given by
\[ T_n : \mor{\set{M}_{w+2}}{\set{M}_{w+2}},  \]
\[ P \mapsto \sum_{M \in \mathcal{HM}_n} P[M].  \]
This operator preserves the subspaces $\set{S}_{w+2}$ and $\set{S}_{w+2}^+$. The key fact that we need is the following theorem.
\begin{theorem}[\cite{Me94}]\label{heckeemb} There exists a perfect pairing 
\[ \langle -,-\rangle : \mor{S_{w+2}(\C) \times (\set{S}_{w+2}^+ \otimes \C)}{\C}     \] 
which is Hecke equivariant, in the sense that $\langle T_n f, P \rangle = \langle f, T_n P\rangle$ for all $n \geq 1$, $f \in S_{w+2}(\C)$, and $P \in \set{S}_{w+2}^+ \otimes \C$.  
\end{theorem}
As a consequence of \cref{heckeemb}, any eigenvalue of $T_n$ on $S_{w+2}(\C)$ must be an eigenvalue of $T_n$ on $\set{S}_{w+2}^+\otimes \C$. 

\begin{proposition}\label{modcongruence} Let $n, m, r$, and $c$ be integers such that $n, m, r \geq 1$. Suppose that for each integer $i$ such that $0 \leq i \leq w$ and $i \equiv 0 \pmod{2}$ there exist
\begin{itemize}
\item integers $\alpha_{1}, \ldots, \alpha_{r}$ and
\item matrices $\gamma_{1},\ldots,\gamma_{r} \in \GL_2\left(\Zmod{p}{m}\right)$ and $\epsilon_{1}, \ldots, \epsilon_{r} \in \{\sigma, \tau\}$
\end{itemize}
(all depending on $i$)
such that 
\[ T_n P_{i,w} - c P_{i,w} \equiv \sum_{j = 1}^{r} \alpha_j \cyc{P_{i,w}}{\gamma_j}{\epsilon_j} \pmod{p^m \Z[][X,Y]_w}  \]
where $P_{i,w} = X^i Y^{w-i}$. Then any eigenvalue $\lambda$ of $T_n$ on $\set{S}_{w+2}^+\otimes \C$ satisfies the congruence
\[ \lambda \equiv c \pmod{p^m \Zc[p]}.  \]
\end{proposition} 
\begin{proof} First, we will show that $p^m \set{M}_{w+2} \cap \set{S}_{w+2}^+ = p^m \set{S}_{w+2}^+$. Clearly, $p^m \set{S}_{w+2}^+ \subset p^m \set{M}_{w+2} \cap \set{S}_{w+2}^+$. Now let $P \in \set{S}_{w+2}^+$ and $Q \in \set{M}_{w+2}$ such that $P = p^m Q$. Applying $\iota^\ast$, we get $p^m \iota^\ast(Q) = p^m Q$. Since $\set{M}_{w+2}$ is torsion-free, we find that $\iota^\ast(Q) = Q$. 
Furthermore, $p^m$ annihilates the image of $Q$ in $\set{M}_{w+2} / \set{S}_{w+2}$, which is also torsion-free (\cite{St07}, \S 8.4). Thus $Q \in \set{S}_{w+2}$. Therefore $Q \in \set{S}_{w+2}^+$ and $P \in p^m \set{S}_{w+2}^+$.  

It is clear from the definition of the involution $\iota^\ast$ that $\set{S}_{w+2}^+$ is generated by the image of 
\[ \{X^i Y^{w-i} - X^{w-i} Y^i : 0 \leq i \leq w.    \}  \]
But for $i$ odd, we have the relation
\[ \cyc{X^i Y^{w-i}}{1}{\sigma} = X^i Y^{w-i} - X^{w-i} Y^i = 0   \]
in $\set{M}_{w+2}$. Therefore $\set{S}_{w+2}^+$ is generated by the image of
\[ \{X^i Y^{w-i} - X^{w-i} Y^i : 0 \leq i \leq w \mbox{ and } i \equiv 0 \pmod{2}    \}.   \]
The assumption now implies that
\[ (T_n - cI)(\set{S}_{w+2}^+) \subset p^m \set{M}_{w+2} \cap \set{S}_{w+2}^+ = p^m \set{S}_{w+2}^+   \]
where $I$ is the identity operator. So the operator $T_n - cI$ acting on $\set{S}_{w+2}^+ \otimes \C$ can be represented by a matrix with entries in $p^m \Z$. Therefore any eigenvalue $\lambda$ of $T_n$ on $\set{S}_{w+2}^+$ satisfies $\lambda - c \in p^m \Zc[p]$. 
\end{proof}
We are ready to describe the algorithm. The input of the algorithm will be
\begin{itemize}
\item a prime $\ell$, 
\item a prime $p$, 
\item an exponent $m$, 
\item integers $c_{w_0}$ for each $w_0 \in 2\Zmod{\phi(p^m)}{}$ (where $\phi$ is the Euler totient function, so that $\phi(p^m) = p^{m-1}(p-1)$),
\end{itemize}
which corresponds to conjectured congruences
\[ a_\ell(f) \equiv c_{w_0} \pmod{p^m \Zc[p]}   \]
for all eigenforms $f$ of level $1$ and weight $w+2$ where $w \equiv w_0 \pmod{\phi(p^m)}$. 
If successful, the algorithm produces a list of propositions. There will be exactly one proposition for each $(i_0, w_0) \in \left(2\Zmod{\phi(p^m)}{}\right)^2$, which will be of the form
\begin{proposition}\label{templateprop} If $P_{i,w} = X^iY^{w-i}$ where $i$ is even, $i \equiv i_0 \pmod{\phi(p^m)}$, and $w \equiv w_0 \pmod{\phi(p^m)}$ then
\[ T_\ell P_{i,w} - c_{w_0} P_{i,w} \equiv \sum_{j= 1}^{r} \alpha_{j} \cyc{P_{i,w}}{\gamma_{j}}{\epsilon_{j}} \pmod{p^m\Z[][X,Y]_w}. \]
\end{proposition} 
where
\begin{itemize}
\item $\alpha_{1}, \ldots, \alpha_{r}$ are explicitly given integers depending on $(i_0, w_0)$,   
\item $\gamma_{1},\ldots,\gamma_{r} \in \GL_2\left(\Zmod{p}{m}\right)$ and $\epsilon_1, \ldots, \epsilon_{r} \in \{\sigma, \tau\}$ are explicitly given matrices depending on $(i_0, w_0)$,
\item $w \in 2\Z[\geq 0]$ and $i \in \{0, \ldots, w\}$ will be variables.
\end{itemize} 
Here \cref{templateprop} is just a template (see \cref{explicitexample} for an explicit example). By \cref{heckeemb} and \cref{modcongruence}, the existence of these propositions would imply the desired congruences. 

Now we will explain the algorithm. Let $G$ be the group of matrices in $\GL_2(\Zmod{p}{m})$ generated by the mod $p^m$ reductions of 
\[ \sigma, \tau,  \mat{\ell}{0}{0}{1}, \mat{\ell}{0}{0}{\ell}. \] 
Fix $(i_0, w_0) \in \left(2\Zmod{\phi(p^m)}{} \right)$ and let $i\geq 0$ and $w \geq 0$ vary within the mod $\phi(p^m)$ congruence classes corresponding respectively to $i_0$ and $w_0$. For each given $i$ and $w$, we put $P_{i,w} = X^i Y^{w-i}$ and we build a list of relations 
\[L_{i,w} = \{\cyc{P_{i,w}}{\gamma}{\sigma}, \cyc{P_{i,w}}{\gamma}{\tau}: \gamma \in G\}.\] 
We then use matrix linear algebra over $\Zmod{p}{m}$ to try to solve the equation expressing $T_\ell P_{i,w} - c_{w_0}P_{i,w}$ as a linear combination of elements of $L_{i,w}$. Suppose that a solution
\[T_\ell P_{i,w} - c_{w_0}P_{i,w} \equiv \sum_{j=1}^{r} \alpha_j \cyc{P_{i,w}}{\gamma_j}{\epsilon_j}  \pmod{p^m\Z[][X,Y]}\]
is found. We want to check whether this solution is ``universal", i.e.\ whether
\[ \tag{$\dagger\dagger$}\label{asolution}T_\ell P_{i',w'} - c_{w_0} P_{i',w'} \equiv \sum_{j=1}^{r} \alpha_j \cyc{P_{i',w'}}{\gamma_j}{\epsilon_j} \pmod{p^m\Z[][X,Y]}\]
for all $i' \equiv i_0 \pmod{\phi(p^m)}$ and $w' \equiv w_0 \pmod{\phi(p^m)}$. This is a polynomial identity that could be checked by hand, but given the amount of identities one needs to check, it is better to automate the verification. We do this by transforming \ref{asolution} into a more ``canonical" form. Both the left and the right hand side of \ref{asolution} are $\Zmod{p}{m}$-linear combinations of terms of the form
\[ (aX + bY)^{i'} (cX + dY)^{w' - i'}    \]We rewrite each such term as
\[ (a X + b Y)^{i_0}(c X + d Y)^{(w_0 - i_0)}(aX + b Y)^{i'-i_0} (c X + dY)^{(w'-w_0)-(i'-i_0)}  \]
and then expand the factor $(a X + b Y)^{i_0}(c X + d Y)^{(w_0 - i_0)}$ to obtain
\[ (aX + bY)^{i'} (cX + dY)^{w' - i'} \]\[\equiv \sum_j \lambda_j X^{e_j} Y^{f_j} (aX + bY)^{i'-i_0} (cX + dY)^{(w' - w_0) - (i'-i_0)} \pmod{p^m \Z[][X,Y]_w}.    \]
The coefficients $\lambda_j$ and exponents $e_j$ and $f_j$ do not depend on $i$ or $w$. Moreover, since $i'-i_0 \equiv w'-w_0 \equiv 0 \pmod{\phi(p^m)}$, we can change the coefficients $a,b,c,d$ so that
\begin{enumerate}[(i)] 
\item $0\leq a, b \leq p^v$, where $v = \max\{1, m - v_p(i-i_0) - v_p(\lambda_j)\}$,
\item $0\leq c, d \leq p^v$, where $v = \max\{1, m - v_p((w-w_0)-(i-i_0)) - v_p(\lambda_j)\}$,
\item if $a$ is a unit, then $a = 1$,
\item if $a$ is not a unit but $b$ is a unit, then $b = 1$,
\item if $c$ is a unit, then $c = 1$,
\item if $c$ is not a unit but $d$ is a unit, then $d = 1$.
\end{enumerate}
After applying the above transformation to every term on both sides and collecting terms, we can rewrite \ref{asolution} as
\[ \sum_j Q_j(X,Y) (a_j X + b_j Y)^{i'-i_0}(c_jX + d_jY)^{(w'-w_0) - (i'-i_0)} \equiv 0 \pmod{p^m \Z[][X,Y]_{w}}   \]
where the polynomials $Q_j$ and the numbers $a_j, b_j, c_j, d_j$ do not depend on $i$ and $w$, the numbers $a_j, b_j, c_j, d_j$ satisfy the conditions (i)-(vi), and the quadruples $(a_j, b_j, c_j, d_j)$ are distinct for distinct values of $j$. Therefore, to verify \ref{asolution}, it is enough for the computer to check that $Q_j = 0$ for all $j$. 

We can now give the algorithm in pseudo-code. 
\begin{list}{$\circ$}{} 
\item $G = \left\langle \sigma, \tau,  \mat{\ell}{0}{0}{1}, \mat{\ell}{0}{0}{\ell} \right\rangle$.
\item For $(i_0, w_0) \in \left(2\Zmod{\phi(p^m)}{} \right)$:
	\begin{list}{$\circ$}{} 
	\item For $i \equiv i_0 \pmod{\phi(p^m)}$ and $w \equiv w_0 \pmod{\phi(p^m)}$:
		\begin{list}{$\circ$}{} 
		\item $P_{i,w} = X^i Y^{w-i}$.
		\item $L_{i,w} = \{\cyc{P_{i,w}}{\gamma}{\tau}, \cyc{P_{i,w}}{\gamma}{\sigma}: \gamma \in G\}$, expressed in terms of $\{X^j Y^{w-j}  \}$.
		\item Express $T_\ell P_{i,w} - c_{w_0}P_{i,w}$ in terms of $\{X^j Y^{w-j}  \}$.
		\item Solve for $Q \in \left(\Zmod{p}{m} \right)[L_{i,w}]$ such that $T_\ell P_{i,w} - c_{w_0}P_{i,w} \equiv Q \pmod{p^m\Z[][X,Y]_w}$. If no solution is found, then the algorithm has failed and we halt.
		\item If $Q$ is a universal solution for $(i_0, w_0)$, break loop and move to next couple $(i_0, w_0)$.
		\end{list}
	\end{list}
\end{list}
This algorithm is not guaranteed to terminate or to find the required solutions. But if it executes successfully, it provides the needed polynomial identities to prove the desired congruence. The code, written in Sage, can be found on the author's website (\cite{Data}), along with computer-generated pdf files containing the propositions of the form \cref{templateprop}, and the corresponding solutions stored as Sage objects. 

This method fails to prove congruences in situations where there is more than one congruence class of eigenforms corresponding to the same congruence class of the weight. For example, when $p = 3$ and $w \equiv 0$ or $4 \pmod{6}$, it is expected that eigenvalues $a_2$ of $T_2$ at weight $w+2$ satisfy $a_2 \equiv 3$ or $6 \pmod{9\Zc[3]}$ and that both congruence classes occur. A similar situation occurs for congruences modulo $2^m \Zc[2]$ when $m \geq 8$. In such a situation, \cref{modcongruence} cannot be applied.

The method can also fail even when there only one predicted congruence class. For example, when $p = 3$ and $w \equiv 2 \pmod{6}$, \cref{a2conj} predicts that $a_2 \equiv 0 \pmod{9\Zc[3]}$. However, it is not true that $T_2(\set{S}_{w+2}^+) \subset 9\set{S}_{w+2}^+$.

\appendix
\section{Example} We will illustrate the method presented in \cref{section:algorithm} by a small example. Let $p = 3$ and consider the operator $T_2$. Hatada's theorem \cref{hatada3} says that the eigenvalues of $T_2$ on level $1$ cuspforms are divisible by $3$. We will prove this using our method. The action of $T_2$ on $P = X^i Y^{w-i}$ is given by
\[ T_2 P = P\left[\gamma_1 + \gamma_2 + \gamma_3 + \gamma_4   \right]  \]
where
\[ \gamma_1 = \mat{2}{0}{0}{1}, \gamma_2 = \mat{1}{0}{0}{2}, \gamma_3 = \mat{2}{1}{0}{1}, \gamma_4 = \mat{1}{0}{1}{2}. \]
The algorithm gives us the following proposition.
\begin{proposition}\label{explicitexample}\leavevmode
\begin{enumerate}[(i)]
\item If $P = X^i Y^{w-i}$ with $i \equiv 0 \pmod{2}$ and $w \equiv 0 \pmod{2}$ then \[ T_{2}P \equiv   2\cyc{P}{\mat{1}{0}{0}{1}}{\sigma}  + \cyc{P}{\mat{0}{1}{2}{0}}{\tau}  \indent \pmod{3\Z[][X,Y]}.\]
\item If $P = X^i Y^{w-i}$ with $i \equiv 1 \pmod{2}$ and $w \equiv 0 \pmod{2}$ then \[ T_{2}P \equiv   \cyc{P}{\mat{1}{0}{0}{1}}{\sigma}  + 2\cyc{P}{\mat{0}{1}{2}{0}}{\tau}  \indent \pmod{3\Z[][X,Y]}.\]
\end{enumerate}
\end{proposition}
Let us prove (ii). Assume $i \equiv w-i \equiv 1 \pmod{2}$. The right hand side is
\[ (X^iY^{w-i})\left[ 1 + \mat{0}{-1}{1}{0} - \mat{0}{1}{2}{0} - \mat{1}{2}{0}{1}  - \mat{2}{0}{1}{2} \right]  \]
\[ = (X)^i (Y)^{w-i} + (-Y)^i (X)^{w-i} - (Y)^i(-X)^{w-i} - (X-Y)^i(Y)^{w-i} - (-X)^i(X-Y)^{w-i}   \]
\[ = XY(X)^{i-1}(Y)^{w-i-1} + (-XY + Y^2)(X-Y)^{i-1}Y^{w-i-1} + (X^2 - XY)(X)^{i-1}(X-Y)^{w-i-1}.  \]
The left hand is
\[ (X^iY^{w-i})\left[\mat{2}{0}{0}{1} + \mat{1}{0}{0}{2} + \mat{2}{1}{0}{1} + \mat{1}{0}{1}{2} \right]  \]
\[ = (-X)^i(Y)^{w-i} + (X)^i(-Y)^{w-i} + (-X+Y)^i(Y)^{w-i} + (X)^i(X-Y)^{w-i}  \]
\[ =  XY(X)^{i-1} Y^{w-i-1} + (-XY + Y^2)(X-Y)^{i-1}(Y)^{w-i-1} + (X^2 - XY)(X)^{i-1}(X-Y)^{w-i-1}.  \]
Thus $(ii)$ is proven. The statement $(i)$ is done similarly. 
\section*{Acknowledgements} The author would like to thank Ian Kiming and Gabor Wiese for many interesting discussions on the topic of eigenforms modulo prime powers over the year, as well as Shaunak Deo and Ming-Lun Hsieh for helpful discussions and comments on this work. The author would also like to thank the anonymous referee for thoroughly reading the manuscript and providing numerous valuable comments
and suggestions, in particular suggesting how to fill an earlier gap in the proof of \cref{modcongruence}.

All relevant computations were done using Sage. This research was supported by a Postdoctoral Fellowship at the National Center for Theoretical Sciences, Taipei, Taiwan.

\providecommand{\bysame}{\leavevmode\hbox to3em{\hrulefill}\thinspace}
\providecommand{\MR}{\relax\ifhmode\unskip\space\fi MR }
\providecommand{\MRhref}[2]{%
  \href{http://www.ams.org/mathscinet-getitem?mr=#1}{#2}
}
\providecommand{\href}[2]{#2}

\end{document}